\documentclass[onefignum,onetabnum]{siamart250211}
\usepackage{braket,amsfonts,amsmath}
\usepackage{booktabs} 
\usepackage[caption=false]{subfig}
\usepackage{tikz-cd}
\usepackage{algorithm}
\usepackage{algpseudocode}
\usepackage{xspace}
\usepackage[mathscr]{euscript}
\usepackage{bold-extra}
\usepackage[most]{tcolorbox}
\usepackage{multirow}
\usepackage{graphicx}
\usepackage{graphics}
\usepackage{float}
\graphicspath{{figures/}{./}}
\usepackage{hyperref}
\usepackage{cleveref}

\newsiamthm{claim}{Claim}
\newsiamremark{remark}{Remark}
\newsiamremark{hypothesis}{Hypothesis}
\crefname{hypothesis}{Hypothesis}{Hypotheses}

\title{Transform-Based Multilinear Algebra via \\ Tensor Decomposition\thanks{Submitted to the editors \today.
\funding{Funding information goes here.}}}

\author{
Yidan Mei\thanks{Department of Statistics and Data Science, Yale University, New Haven, CT 06511, USA (\email{yidan.mei@yale.edu}).}
\and Shenghan Mei\thanks{Department of Mathematics, University of North Carolina at Chapel Hill, Chapel Hill, NC 27599, USA (\email{shmei@unc.edu}).}
\and Ziqin He\thanks{Department of Mathematics, University of North Carolina at Chapel Hill, Chapel Hill, NC 27599, USA (\email{zhe21@unc.edu}).}
\and Can Chen\thanks{School of Data and Information Sciences, Department of Mathematics, and Department of Biostatistics, University of North Carolina at Chapel Hill, Chapel Hill, NC 27599, USA (\email{canc@unc.edu}).}}

\begin{document}
\maketitle

\begin{abstract}
Transform-based tensor products, including the T-product and its more general form, namely the higher-order tensor-tensor product, have become fundamental tools for multilinear data analysis in applications such as image processing, signal reconstruction, and robotics. While invertible transforms enable tensor computations to be carried out via matrix operations in the transform domain, the resulting storage and computational costs remain prohibitive for high-dimensional, higher-order tensors. To address this challenge, we integrate low-rank tensor decomposition techniques, specifically tensor train decomposition (TTD) and hierarchical Tucker decomposition (HTD), into transform-based multilinear algebra to improve computational and memory efficiency. In particular, we develop TTD- and HTD-based formulations for the T-product and its associated key algebra, such as block diagonalization and tensor singular value decomposition, by operating directly on the factor matrices or tensors of the decompositions. The framework is further generalized to the higher-order tensor-tensor product and applied to multilinear model order reduction problems. We demonstrate the effectiveness and efficiency of our framework with numerical examples. 
\end{abstract}
	
	

\section{Introduction}
Multidimensional data have become ubiquitous in modern scientific and engineering applications, including image and video processing \cite{ahmadi2023fast,woods2011multidimensional,xia2024tensor}, biomedical measurement analysis \cite{azam2022review,sedighin2024tensor}, robotics \cite{shetty2024tensor,wijesoma2006toward}, and networked dynamical systems \cite{chen2022explicit,chen2021controllability,shen2017tensor}. In these applications, the data are inherently multiway, with spatial, temporal, spectral, or modal relationships encoded across multiple indices. For example, video sequences contain coupled spatial and temporal correlations, while networked and biomedical systems often exhibit interactions across multiple functional or physical modes. Traditional approaches frequently reshape tensors into vectors or matrices through unfolding in order to apply classical linear algebra techniques. Although this may simplify computations, it often destroys intrinsic multilinear structure, weakens correlations across modes, and leads to redundant high-dimensional representations \cite{chen2021multilinear,ma2026online,rogers2013multilinear}. These limitations become increasingly pronounced as the size and order of the data grow. Multilinear algebra therefore provides a natural framework for modeling and processing such data while preserving their underlying multiway structure, enabling more compact representations and more effective exploitation of correlations across multiple dimensions simultaneously.

Among existing tensor frameworks, transform-based tensor algebra, particularly the T-product, provides an effective operator-based framework for third-order tensors \cite{braman2010third,cao2023some,kilmer2013third,kilmer2011factorization,lund2020tensor,miao2021t}. By combining block-circulant representations with invertible transforms, the T-product induces tensor analogues of many classical matrix operations, including inversion, eigendecomposition, and tensor singular value decomposition (T-SVD), while enabling these operations to be computed through decoupled matrix problems in the transform domain. Owing to its elegant algebraic structure and computational efficiency, the T-product framework has been successfully applied to a wide range of problems, including pose estimation and recognition \cite{hao2013facial,hoover2011pose}, tensor completion and robust principal component analysis \cite{he2022tensor,lu2019tensor,zhang2016exact}, multi-view learning \cite{wang2026towards, wu2019essential,yin2018multiview}, signal processing \cite{11430563,pena2023t,saibaba2025tensor}, tensor linear systems \cite{ma2022randomized}, and multilinear systems and control \cite{ding2025tensor,he2026data,hoover2021new,ma2025multilinear,mao2025data,mei2026data,mei2025controllability}. Motivated by the increasing prevalence of higher-order data, several generalizations of transform-based tensor algebra have been developed beyond the third-order setting. In particular, the higher-order tensor-tensor product framework generalizes the T-product to tensors of arbitrary order by applying separable transforms along multiple modes, thereby reducing tensor operations to collections of independent matrix computations in the transform domain while preserving the underlying multiway structure \cite{kernfeld2015tensor, martin2013order,ozdemir2022high,wang2022hot}.

Although the T-product and related transform-based tensor algebra have enabled a broad range of applications, they do not fully eliminate the underlying computational bottleneck. For large-scale problems, both storage and runtime can still become prohibitive because the computations rely on full tensor representations, block-circulant embeddings, and repeated dense matrix factorizations in the transform domain. These costs grow further as either the tensor order or individual mode dimensions increase, making scalability a critical concern. The issue is particularly pronounced in high-dimensional multilinear systems analysis and control, where tensor products, spectral computations, and decomposition-based routines must often be performed repeatedly within iterative algorithms \cite{he2026data,mao2025data,mei2026data}. Similar challenges also arise in high-dimensional imaging and scientific computing, where large mode sizes lead to substantial memory demands and computational overhead \cite{bachmayr2018parametric,bengua2017efficient,chang2020weighted}.

To address this challenge, it is natural to combine transform-based multilinear algebra with low-rank tensor decompositions. A substantial body of work has shown that tensor decompositions, including CANDECOMP/PARAFAC decomposition \cite{bader2008efficient,kolda2009tensor}, higher-order singular value decomposition \cite{5447070,de2000multilinear}, tensor train decomposition (TTD) \cite{oseledets2011tensor,oseledets2009breaking}, and hierarchical Tucker decomposition (HTD) \cite{grasedyck2010hierarchical}, can represent high-dimensional tensors in a compact form while preserving essential multilinear structure \cite{tokcan2026tensor}. In this work, we focus on TTD and HTD, which are particularly well suited for higher-order problems due to their numerically stable and storage-efficient representations and their ability to support key numerical operations, such as orthogonalization, truncation, and linear-system-related computations, directly on low-dimensional core tensors or hierarchical factors rather than on full tensors. Moreover, both formats have demonstrated strong performance in large-scale computational settings and have been successfully employed in dynamical low-rank approximation for time-dependent tensor problems \cite{lubich2013dynamical,mao2025tensor}. These characteristics make TTD and HTD particularly attractive for accelerating transform-based tensor operations in regimes where both dimensionality and computational cost are substantial.

The primary contribution of this work is a tensor decomposition-based computational framework for transform-based multilinear algebra. We reformulate the T-product and related operations directly within TTD and HTD formats, so that the dominant computations are performed on compressed core tensors or factors without explicitly forming or reconstructing full tensors. In particular, we derive TTD- and HTD-based formulations for block diagonalization, the T-product, and T-SVD, and provide a detailed computational complexity analysis of the resulting algorithms. We further extend these constructions from the third-order setting to the higher-order tensor-tensor product framework. Finally, we demonstrate the effectiveness and efficiency of the proposed approach through numerical examples and illustrate its applicability to model order reduction in multilinear dynamical systems. To support reproducibility and practical use, we also provide MATLAB codes implementing the main decomposition-based tensor operations developed in this work.

The remainder of this article is organized as follows. \Cref{sec:preliminaries} reviews tensor preliminaries, including the T-product, the higher-order tensor-tensor product, and the fundamentals of TTD and HTD. \Cref{sec:3rdoperation} develops TTD- and HTD-based formulations for third-order transform-based tensor algebra, including block diagonalization, the T-product, and T-SVD. \Cref{sec:higherorderoperation} extends these constructions to higher-order tensor-tensor products. \Cref{sec:numeg} presents numerical experiments illustrating the performance of the proposed methods and their application to multilinear  model order reduction. \Cref{sec:conclusion} concludes with future research directions.

\section{Preliminaries} \label{sec:preliminaries}
Tensors provide a natural representation for multidimensional data, extending vector and matrix algebra to higher orders \cite{chen2024tpds,kolda2009tensor}. The order of a tensor refers to the number of its dimensions, with each dimension referred to as a mode. A $k$th-order tensor is typically denoted by $\mathscr{T} \in \mathbb{R}^{n_1 \times n_2 \times \cdots \times n_k}$. Many matrix operations and decompositions can be extended to the tensor setting in a consistent manner. For convenience, we adopt MATLAB colon notation “:” to denote all entries along a given tensor dimension. For example, the frontal slices, which are matrices obtained by fixing all indices except the first two, of a third-order tensor $\mathscr{T}\in\mathbb{R}^{n_1\times n_2\times n_3}$ can be written as $\mathscr{T}(:,:,j_3)$ for $j_3=1,2,\dots,n_3$.

\subsection{T-product} The T-product extends matrix multiplication to third-order tensors by enabling tensor multiplication through circular convolution operations \cite{braman2010third,kilmer2013third,kilmer2011factorization}. A key component of this framework is the use of block-circulant operators, which represent a tensor as a structured block matrix whose multiplication can be performed using standard matrix operations. Complementing this construction, the unfolding operator reshapes a tensor into a block column matrix,  establishing a bridge between tensor representations and classical matrix multiplication. Given a third-order tensor $\mathscr{T}\in\mathbb{R}^{n_1\times n_2\times n_3}$, the two operators are defined as 
\begin{align*}
\texttt{bcirc}(\mathscr{T}) &= 
\begin{bmatrix}
\mathscr{T}(:,:,1) & \mathscr{T}(:,:,n_3) & \cdots & \mathscr{T}(:,:,2)\\
\mathscr{T}(:,:,2) & \mathscr{T}(:,:,1) & \cdots & \mathscr{T}(:,:,3)\\
\vdots & \vdots & \ddots & \vdots\\
\mathscr{T}(:,:,n_3) & \mathscr{T}(:,:,n_3-1) & \cdots & \mathscr{T}(:,:,1)
\end{bmatrix}\in\mathbb{R}^{n_1n_3\times n_2n_3}, \\
\texttt{unfold}(\mathscr{T}) &= 
\begin{bmatrix}
\mathscr{T}(:,:,1)^\top &
\mathscr{T}(:,:,2)^\top &
\cdots &
\mathscr{T}(:,:,n_3)^\top
\end{bmatrix}^\top\in\mathbb{R}^{n_1n_3\times n_2}.
\end{align*}
The reverse operation $\texttt{fold}(\cdot)$ satisfies $\texttt{fold}(\texttt{unfold}(\mathscr{T}))=\mathscr{T}$.  
With these operators, the third-order T-product is defined as follows. 

\begin{definition} [T-product] \label{def:3rdtprod}
The T-product between two third-order tensors $\mathscr{T}\in\mathbb{R}^{n_1\times l\times n_3}$ and $\mathscr{S}\in\mathbb{R}^{l\times n_2\times n_3}$, denoted by $\mathscr{T}\star\mathscr{S}$, is defined as 
\begin{equation*}
\mathscr{T}\star\mathscr{S} = \texttt{fold}\big(\texttt{bcirc}(\mathscr{T})\texttt{unfold}(\mathscr{S})\big) \in\mathbb{R}^{n_1\times n_2\times n_3}.
\end{equation*}
\end{definition}

The T-product naturally induces tensor analogues of standard matrix concepts. A tensor $\mathscr{I}\in \mathbb{R}^{n\times n \times n_3}$ is said to be the T-identity tensor if its first frontal slice $\mathscr{I}(:,:,1)$ is an identity matrix, and the following slices $\mathscr{I}(:,:,j_3)$,  $j_3 = 2,3,\dots, n_3$ are all zeros.  The T-transpose of a tensor $\mathscr{T} \in \mathbb{R}^{n_1 \times n_2 \times n_3}$, denoted by $\mathscr{T}^\top$, is defined as transposing each frontal slice $\mathscr{T}(:,:,j_3)$ and then reversing the order of the transposed slices from 2 to $n_3$. The T-inverse of a tensor $\mathscr{T} \in \mathbb{R}^{n \times n \times n_3}$, denoted by $\mathscr{T}^{-1}$, is defined such that $\mathscr{T} \star \mathscr{T}^{-1} = \mathscr{T}^{-1} \star \mathscr{T} = \mathscr{I}$, where $\mathscr{I}$ is the identity tensor. The tensor $\mathscr{T} \in \mathbb{R}^{n \times n \times n_3}$ is said to be T-orthogonal if $\mathscr{T} \star \mathscr{T}^\top = \mathscr{T}^\top \star \mathscr{T} = \mathscr{I}$. A tensor $\mathscr{T} \in \mathbb{R}^{n \times n \times n_3}$ is called F-diagonal if each frontal slice $\mathscr{T}(:,:,j_3)$ is a diagonal matrix. Note that we use the same notations for matrix and T-product
operations (e.g., transpose, inverse, etc.)
whenever no ambiguity arises. More importantly,  the classical singular value decomposition (SVD) admits a natural extension to tensors within the T-product framework, known as  tensor singular value decomposition (T-SVD) \cite{zhang2018randomized,zhang2016exact}.

\begin{definition}[T-SVD]
The T-SVD of a tensor $\mathscr{T}\in\mathbb{R}^{n_1\times n_2\times n_3}$ is defined as
\begin{equation}
\mathscr{T}=\mathscr{U}\star\mathscr{S}\star\mathscr{V}^{\top},
\end{equation}
where $\mathscr{U}\in\mathbb{R}^{n_1\times n_1\times n_3}$ and $\mathscr{V}\in\mathbb{R}^{n_2\times n_2\times n_3}$ are T-orthogonal, and $\mathscr{S}\in\mathbb{R}^{n_1\times n_2\times n_3}$ is an F-rectangular diagonal tensor. The tubes $\mathscr{S}_{jj:}\in\mathbb{R}^{n_3}$ are referred to as the singular tubes of $\mathscr{T}$ for $j=1,2,\dots,\min{\{n_1,n_2\}}$.
\end{definition}

T-SVD can be computed  by exploiting the discrete Fourier transform together with the classical matrix SVD. In particular, a block-circulant matrix can be block-diagonalized through left and right multiplication by block Fourier matrices. For a third-order tensor $\mathscr{T}\in\mathbb{R}^{n_1\times n_2\times n_3}$, the Fourier transform of its block-circulant representation, denoted by $\mathcal{F}{\{\texttt{bcirc}(\mathscr{T})\}}$, is expressed as
\begin{equation}\label{eq:fourier_def}
\scalebox{0.98}{$\mathcal{F}\{\texttt{bcirc}(\mathscr{T})\} = (\textbf{F}_{n_3} \otimes \textbf{I}_{n_1}) \, \texttt{bcirc}(\mathscr{T}) \, (\textbf{F}_{n_3}^{*} \otimes \textbf{I}_{n_2})
= \texttt{blkdiag}(\textbf{T}_1,\textbf{T}_2, \dots, \textbf{T}_{n_3}),$}
\end{equation}
where $\texttt{blkdiag}(\cdot)$ denotes the MATLAB block diagonal operator, $\otimes$ represents the Kronecker product, $*$ denotes the conjugate transpose, 
$\textbf{I}_{n_1}\in\mathbb{R}^{n_1\times n_1}$ is the identity matrix, and $\textbf{F}_{n_3} \in \mathbb{C}^{n_3 \times n_3}$ is the discrete Fourier transform matrix defined as
\begin{equation*}
\textbf{F}_{n_3} = \frac{1}{\sqrt{n_3}} 
\begin{bmatrix}
1 & 1 & 1 & \cdots & 1 \\
1 & \omega & \omega^2 & \cdots & \omega^{-1} \\
\vdots & \vdots & \vdots & \ddots & \vdots \\
1 & \omega^{n_3-1} & \omega^{2(n_3-1)} & \cdots & \omega^{(n_3-1)^2}
\end{bmatrix},
\end{equation*}
with $\omega = \exp(-\frac{2\pi i}{n_3})$ ($i$ is the imaginary number here). Next, the matrix SVD of each diagonal block matrix is computed as $\textbf{T}_{j_3}=\textbf{U}_{j_3}\textbf{S}_{j_3}\textbf{V}_{j_3}^*$ for $j_3=1,2,\dots, n_3$. Finally, the factor tensors $\mathscr{U}$, $\mathscr{S}$, and $\mathscr{V}$ are obtained by applying the inverse discrete Fourier transform to the collections of matrices ${\textbf{U}_{j_3}}$, ${\textbf{S}_{j_3}}$, and ${\textbf{V}_{j_3}}$, respectively. It is worth noting that the T-SVD of $\mathscr{T}$ differs from the matrix SVD of $\texttt{bcirc}(\mathscr{T})$ since the block circulant matrix $\texttt{bcirc}(\mathscr{S})$ is not
diagonal.

\subsection{Higher-order tensor-tensor product} The higher-order tensor-tensor product, introduced in \cite{ozdemir2022high}, generalizes the T-product to tensors of arbitrary order,  enabling transform-based tensor operations to be performed in the higher-order setting. To define the higher-order tensor-tensor product, we first introduce two notions of tensor multiplication.

\begin{definition} [Tensor-matrix multiplication]
    The mode-$p$ tensor-matrix multiplication for a $k$th-order tensor $\mathscr{T}\in\mathbb{R}^{n_1\times n_2\times \cdots \times n_k}$ and a matrix $\textbf{A}\in\mathbb{R}^{m\times n_p}$, denoted by $\mathscr{T}\times_p\textbf{A}\in\mathbb{R}^{n_1\times n_2\times \cdots \times n_{p-1}\times m\times n_{p+1}\times\cdots\times n_k}$, is defined as
\begin{equation*}
\scalebox{0.95}{$\displaystyle(\mathscr{T}\times_p \textbf{A}) (j_1,j_2,\dots, j_{p-1},i,j_{p+1},\dots, j_k) =
\sum_{j_p=1}^{n_p}
\textbf{A}(i, j_p)
\mathscr{T}(j_1,j_2,\dots, j_{p-1}, j_p, j_{p+1},\dots, j_k).$}
\end{equation*}
\end{definition}

\begin{definition} [Facewise product]
For $k$th-order tensors
$\mathscr{T}\in\mathbb{R}^{n_1\times \ell\times n_3\times \cdots\times n_k}$ and $\mathscr{S}\in\mathbb{R}^{\ell\times n_2\times n_3\times \cdots\times n_k}$, the facewise product between the two tensors, denoted by $\mathscr{T} \ \Delta \ \mathscr{S}\in\mathbb{R}^{n_1\times n_2\times n_3\times n_4\times \cdots\times n_k}$, is defined as
\begin{equation*}
({\mathscr{T}}\ \Delta\ {\mathscr{S}})(:,:,j_3, j_4,\dots, j_k)
=
{\mathscr{T}}(:,:,j_3,j_4, \dots, j_k)\,
{\mathscr{S}}(:,:,j_3,j_4, \dots, j_k),
\end{equation*}
for every index tuple $j_3,j_4,\dots,j_k$.
\end{definition}

\begin{definition} [Higher-order tensor-tensor product]
For two $k$th-order tensors
$\mathscr{T}\in\mathbb{R}^{n_1\times \ell\times n_3\times \cdots\times n_k}$ and $\mathscr{S}\in\mathbb{R}^{\ell\times n_2\times n_3\times \cdots\times n_k}$,
the higher-order tensor-tensor product between the two tensors, denoted by $\mathscr{T}\star_H \mathscr{S}\in\mathbb{R}^{n_1\times n_2\times n_3\times n_4\times \cdots\times n_k}$, is defined  as
\begin{equation}
\label{eq:kthorderTprod}
\mathscr{T}\star_H \mathscr{S}
=
\bigl(\widetilde{\mathscr{T}}\ \Delta\ \widetilde{\mathscr{S}}\bigr)
\times_3 \textbf{F}_{n_3}^{-1}
\times_4 \textbf{F}_{n_4}^{-1}
\times_5\cdots
\times_k \textbf{F}_{n_k}^{-1},
\end{equation}
where $\widetilde{\mathscr{T}}
=
\mathscr{T}
\times_3 \textbf{F}_{n_3}
\times_4 \textbf{F}_{n_4}
\times_5\cdots
\times_k \textbf{F}_{n_k}$ and $
\widetilde{\mathscr{S}}
=
\mathscr{S}
\times_3 \textbf{F}_{n_3}
\times_4 \textbf{F}_{n_4}
\times_5\cdots
\times_k \textbf{F}_{n_k}$ are obtained by applying the discrete Fourier transform matrices along modes $3$ through $k$ via tensor-matrix multiplication.
\end{definition}


Applying the discrete Fourier transform along modes $3$ through $k$ decouples the
computation into $n_3n_4\cdots n_k$ independent matrix computations in the Fourier
domain, thereby extending T-product-based algebraic structures such as T-SVD to
$k$th-order tensors. More generally, any invertible linear transform can replace the
discrete Fourier transform, yielding a broader class of transform-based tensor-tensor products.

\subsection{Tensor decomposition} Tensor decomposition is a fundamental technique for representing high-dimensional tensors using lower dimensional components \cite{kolda2009tensor}. Among the many tensor decomposition formats, this work focuses on tensor train decomposition (TTD) \cite{oseledets2011tensor} and hierarchical Tucker decomposition (HTD) \cite{grasedyck2010hierarchical} because of their favorable numerical stability and strong compression capability.

\begin{definition} [TTD]
    Let $\mathscr{T}\in\mathbb{R}^{n_1\times n_2\times\cdots\times n_k}$ be a $k$th-order tensor. The TTD of $\mathscr{T}$ is defined as
\begin{equation*}
\mathscr{T}(j_1,j_2,\dots,j_k)
=
\sum_{\alpha_0=1}^{r_0}\sum_{\alpha_1=1}^{r_1}\cdots \sum_{\alpha_k=1}^{r_k}
\mathscr{G}_1(\alpha_0,j_1,\alpha_1)
\mathscr{G}_2(\alpha_1,j_2,\alpha_2)
\cdots
\mathscr{G}_k(\alpha_{k-1},j_k,\alpha_k),
\end{equation*}
where $\mathscr{G}_t\in\mathbb{R}^{r_{t-1}\times n_t\times r_t}$, $t=1,2,\dots,k$, are referred to as the core tensors, and $\{r_t\}_{t=0}^{k}$ are called the TT-ranks with $r_0=r_k=1$.
\end{definition} 

\begin{definition}[HTD]
Let $\mathscr{T}\in\mathbb{R}^{n_1\times n_2\times\cdots\times n_k}$ be a $k$th-order tensor, and let  $\mathcal{T}$ be a dimension tree associated with the index set $D=\{1,2,\dots,k\}$. The HTD of $\mathscr{T}$ is defined recursively on $\mathcal{T}$ as
\begin{equation*}
\textbf{U}_t(:,\alpha_t)=
\sum_{\alpha_{t_1}=1}^{r_{t_1}}
\sum_{\alpha_{t_2}=1}^{r_{t_2}}
\mathscr{B}_t(\alpha_{t_1},\alpha_{t_2},\alpha_t)
\textbf{U}_{t_1}(:,\alpha_{t_1})
\otimes
\textbf{U}_{t_2}(:,\alpha_{t_2}),
\end{equation*}
where $t_1$ and $t_2$ denote the two child nodes of node $t$ in the dimension tree $\mathcal{T}$, $\mathscr{B}_t\in
\mathbb{R}^{r_{t_1}\times r_{t_2}\times r_t}$ is  the transfer tensor associated with node $t$, and $r_t$ is the corresponding hierarchical rank. For each leaf node $t\in\{1,2,\dots,k\}$, the factor matrices satisfy $\textbf{U}_{t}\in\mathbb{R}^{n_t\times r_t}$. The original tensor $\mathscr{T}$ is recovered recursively by contracting the transfer tensors and factor matrices from the leaf nodes to the root of the dimension tree.
\end{definition}

It is worth noting that for third-order tensors, HTD can be equivalently written as the standard Tucker decomposition (TD), up to a reparameterization of the core tensor. Specifically, given a third-order tensor $\mathscr{T}\in\mathbb{R}^{n_1\times n_2\times n_3}$, its TD  is defined as
\begin{equation*}
    \mathscr{T}= \mathscr{G}\times_1\textbf{U}_1\times_2\textbf{U}_2\times_3 \textbf{U}_3,
\end{equation*}
where $\mathscr{G}\in\mathbb{R}^{r_1\times r_2\times r_3}$ is referred to as the core tensor, and $\textbf{U}_t\in\mathbb{R}^{n_t\times r_t}$, $t=1,2,3$, are the factor matrices. In this case, the HTD of $\mathscr{T}$ reduces to a single-level dimension tree, and its transfer tensors can be absorbed into a single Tucker core tensor, thereby yielding an equivalent representation. We will adopt the standard Tucker format in place of HTD for the development of the T-product algebra.

\section{TTD- and HTD-based T-product algebra}\label{sec:3rdoperation}
In this section, we develop decomposition-based formulations for the fundamental operations of the T-product algebra, including block diagonalization, the T-product, and T-SVD, by exploiting the low-rank structures provided by TTD and HTD. Specifically, we reformulate these operations directly in terms of the decomposition factors,  enabling efficient computation without reconstructing the full tensors.

\subsection{Block diagonalization}
Block diagonalization plays a central role in T-product algebra, as it transforms the block-circulant structure of a third-order tensor into a collection of independent matrices in the Fourier domain. Specifically, given a third-order tensor $\mathscr{T}\in\mathbb{R}^{n_1\times n_2\times n_3}$, we compute its diagonal block  matrices $\textbf{T}_{j_3}\in\mathbb{C}^{n_1\times n_2}$, $j_3=1,2,\dots,n_3$, as defined in \cref{eq:fourier_def}. Although the fast Fourier transform provides a direct and efficient way to perform this operation, the computation can become expensive for large-scale tensors. Notably, if a low-rank TTD or HTD of $\mathscr{T}$ is available, the diagonal block matrices $\textbf{T}_{j_3}$ can be constructed more efficiently by exploiting the factorized structure of the tensor.

\begin{proposition}[TTD-based block diagonalization]\label{prop:ttdbd}
    Let $\mathscr{T}\in\mathbb{R}^{n_1\times n_2\times n_3}$ 
be a third-order tensor in the TTD format with core tensors
$\mathscr{G}_1\in\mathbb{R}^{1\times n_1\times r_1}$,
$\mathscr{G}_2\in\mathbb{R}^{r_1\times n_2\times r_2}$, and 
$\mathscr{G}_3\in\mathbb{R}^{r_2\times n_3\times 1}$. Then the diagonal block  matrices of $\mathscr{T}$ in the Fourier domain are computed as
\begin{align}\label{eq:ttdbd}
\textbf{T}_{j_3}(j_1,j_2)
= 
\sum_{\alpha_1=1}^{r_1}\sum_{\alpha_2=1}^{r_2} \mathscr{G}_1(j_1,\alpha_1)\, \mathscr{G}_2(\alpha_1,j_2,\alpha_2)\, \widetilde{\mathscr{G}}_3(\alpha_2,j_3), 
\end{align}
where $\widetilde{\mathscr{G}}_3 = \mathscr{G}_3 \times_2 \textbf{F}_{n_3}$, for $j_3=1,2,\dots, n_3$.
\end{proposition}
\begin{proof}
   From the definition of block diagonalization in the Fourier domain, the $j$th block diagonal matrix is obtained by applying the discrete Fourier transform along the third mode, equivalently, by multiplying the discrete Fourier matrix $\textbf{F}_{n_3}$ along the third mode of $\mathscr{T}$. Substituting the TTD of $\mathscr{T}$ into this transformation yields
    \begin{align*}
\textbf{T}_{j_3}(j_1,j_2)
=
\sum_{\alpha_1=1}^{r_1}
\sum_{\alpha_2=1}^{r_2}
\mathscr{G}_1(j_1,\alpha_1)
\mathscr{G}_2(\alpha_1,j_2,\alpha_2)
\Bigg(
\sum_{j=1}^{n_3}
\mathscr{G}_3(\alpha_2,j)
\textbf{F}_{n_3}(j,j_3)\Bigg).
\end{align*}
Define $\widetilde{\mathscr{G}}_3 = \mathscr{G}_3 \times_2 \textbf{F}_{n_3}$, and the desired result then follows immediately.
\end{proof}

\begin{proposition}[HTD/TD-based block diagonalization]\label{prop:htdbd}
    Let $\mathscr{T}\in\mathbb{R}^{n_1\times n_2\times n_3}$
be a third-order tensor in the HTD/TD format with core tensor
$\mathscr{G}\in\mathbb{R}^{r_1\times r_2\times r_3}$,
and factor matrices $\textbf{U}_1\in\mathbb{R}^{n_1\times r_1}$, $\textbf{U}_2\in\mathbb{R}^{n_2\times r_2}$, and $\textbf{U}_3\in\mathbb{R}^{n_3\times r_3}$. Then the diagonal block  matrices of $\mathscr{T}$ in the Fourier domain are computed as
\begin{align}\label{eq:htdbd}
\textbf{T}_{j_3}(j_1,j_2)
= 
\sum_{\alpha_1=1}^{r_1}\sum_{\alpha_2=1}^{r_2} \sum_{\alpha_3=1}^{r_3}\mathscr{G}(\alpha_1,\alpha_2,\alpha_3)\textbf{U}_1(j_1,\alpha_1)\textbf{U}_2(j_2,\alpha_2)\widetilde{\textbf{U}}_3(j_3,\alpha_3),
\end{align}
where $\widetilde{\textbf{U}}_3=\textbf{F}_{n_3}\textbf{U}_3$, for $j_3=1,2,\dots,n_3$.
\end{proposition}
\begin{proof}
The proof follows immediately from the observation that multiplying $\mathscr{T}$ by $\textbf{F}_{n_3}$ along the third mode is equivalent to replacing the factor matrix $\textbf{U}_3$ with $\textbf{F}_{n_3}\textbf{U}_3$ in the HTD/TD of $\mathscr{T}$.
\end{proof}


Both propositions provide efficient procedures for constructing the diagonal blocks when the TT- or hierarchical ranks are small relative to the tensor dimensions. The computations can be further accelerated by matricizing intermediate contractions. For the TTD-based approach, reshaping the core tensors allows the contractions to be performed using matrix-vector and matrix-matrix multiplications rather than entrywise tensor contractions. The same strategy applies to the HTD/TD-based approach.

\begin{remark}
    Assume $n_1 = n_2 = n_3 = n$ and let $r$ denote the maximum  TT-rank or hierarchical rank. The computational costs of TTD-based and HTD/TD-based block diagonalization  are estimated as $\mathcal{O}(rn \log n + n(r^2n  + rn^2))$ and $\mathcal{O}(rn\log n  + n(r^3 + r^2n  + rn^2))$, respectively. In contrast, the definition-based approach first constructs the full tensor and then applies the fast Fourier transform along the third mode, which costs $\mathcal{O}(n^3 \log n)$ and requires storing the entire tensor. 
\end{remark}

\subsection{T-product} Using the block-diagonal representation, the T-product can be computed by performing independent matrix multiplications on the corresponding diagonal blocks in the Fourier domain, followed by the inverse Fourier transform to recover the resulting tensor.

\begin{corollary}\label{coro:ttdtproduct}
    Given two third-order tensors 
$\mathscr{T}\in\mathbb{R}^{n_1\times l\times n_3}$ and
$\mathscr{S}\in\mathbb{R}^{l\times n_2\times n_3}$ in the TTD format with core tensors $\mathscr{G}_1$, $\mathscr{G}_2$, $\mathscr{G}_3$ and $\mathscr{H}_1$, 
$\mathscr{H}_2$,
$\mathscr{H}_3$, respectively, their T-product $\mathscr{T}\star\mathscr{S}\in\mathbb{R}^{n_1\times n_2\times n_3}$ is computed as
\begin{equation*}
   \mathscr{T}\star\mathscr{S} = \texttt{bcirc}^{-1}\Big(\mathcal{F}^{-1}(\texttt{blkdiag}\big(\textbf{T}_1\textbf{S}_1,\textbf{T}_2\textbf{S}_2,\dots,\textbf{T}_{n_3}\textbf{S}_{n_3})\big)\Big),
\end{equation*}
where $\textbf{T}_{j_3}$ and $\textbf{S}_{j_3}$ denote the 
$j_3$th diagonal blocks  of $\mathscr{T}$ and $\mathscr{S}$ in the Fourier domain, respectively, computed from their TTD representations according to \eqref{eq:ttdbd}.
\end{corollary}
\begin{proof}
    The proof follows immediately from the property of the T-product under the Fourier domain and \cref{prop:ttdbd}.
\end{proof}

Similarly, for tensors represented in the HTD/TD format, the T-product can be computed using the block-diagonal representation established in \cref{prop:htdbd}.  However, this approach does not fully exploit the underlying low-rank TTD or HTD/TD structure, as it requires the explicit formation of intermediate diagonal blocks. A more efficient formulation can be obtained by expressing the T-product directly in terms of the TTD or HTD/TD factors of the input tensors, thereby preserving the compressed structure throughout the computation.

\begin{proposition}[TTD-based T-product]\label{prop:ttd_tprod}
Given two third-order tensors 
$\mathscr{T}\in\mathbb{R}^{n_1\times l\times n_3}$ and
$\mathscr{S}\in\mathbb{R}^{l\times n_2\times n_3}$ in the TTD format with core tensors $\mathscr{G}_1$, $\mathscr{G}_2$, $\mathscr{G}_3$ and $\mathscr{H}_1$, 
$\mathscr{H}_2$,
$\mathscr{H}_3$, and TT-ranks $\{1,r_1,r_2,1\}$ and $\{1,s_1,s_2,1\}$, respectively, their T-product $\mathscr{T}\star\mathscr{S}\in\mathbb{R}^{n_1\times n_2\times n_3}$ admits a TTD representation with core tensors
\begin{align*}
     &\mathscr{P}_1(1,j_1,\phi(\alpha_1,\beta_1)) = \mathscr{G}_1(1,j_1,\alpha_1), \\
    &\mathscr{P}_2(\phi(\alpha_1,\beta_1),j_2,\phi(\alpha_2,\beta_2)) = \sum_{j=1}^{l}
    \mathscr{G}_2(\alpha_1,j,\alpha_2)\,
    \mathscr{H}_1(1,j,\beta_1)\,
    \mathscr{H}_2(\beta_1,j_2,\beta_2), \\
    &\mathscr{P}_3 = \widetilde{\mathscr{P}}_3 \times_2 \textbf{F}_{n_3}^{-1} \text{ with } \widetilde{\mathscr{P}}_3(\phi(\alpha_2,\beta_2),j_3,1) =
    \widetilde{\mathscr{G}}_3(\alpha_2,j_3,1)\,
    \widetilde{\mathscr{H}}_3(\beta_2,j_3,1),
\end{align*} 
where $\phi(\alpha_t,\beta_t)=\alpha_t+(\beta_t-1)r_t$ for $t=1,2$, $\widetilde{\mathscr{G}}_3=\mathscr{G}_3\times_2\textbf{F}_{n_3}$, and
$\widetilde{\mathscr{H}}_3=\mathscr{H}_3\times_2\textbf{F}_{n_3}$,  and with TT-ranks $\{1,r_1s_1,r_2s_2,1\}$.
\end{proposition}
\begin{proof}
    Denote by $\mathscr{C}=\mathscr{T}\star\mathscr{S}$  the T-product and by $\widetilde{\mathscr{C}}$  its representation in the Fourier domain. Since $\widetilde{\mathscr{C}}(j_1,j_2,j_3) = \sum_{j=1}^{l}\textbf{T}_{j_3}(j_1,j)\textbf{S}_{j_3}(j,j_2)$ for $j_3=1,2,\dots, n_3$, substituting $\textbf{T}_{j_3}(j_1,j)$ and $\textbf{S}_{j_3}(j,j_2)$ with the corresponding core representations according to \eqref{eq:ttdbd} yields
    \begin{align*}
\widetilde{\mathscr{C}}(j_1,j_2,j_3)&= 
\sum_{j=1}^{l}
\sum_{\alpha_1=1}^{r_1}
\sum_{\alpha_2=1}^{r_2}
\sum_{\beta_1=1}^{s_1}
\sum_{\beta_2=1}^{s_2}
\mathscr{G}_1(1,j_1,\alpha_1)
\mathscr{G}_2(\alpha_1,j,\alpha_2)
\widetilde{\mathscr{G}}_3(\alpha_2,j_3,1)\\
& \times\mathscr{H}_1(1,j,\beta_1)\, \mathscr{H}_2(\beta_1,j_2,\beta_2)\, \widetilde{\mathscr{H}}_3(\beta_2,j_3,1) 
    \end{align*}
Rearranging the terms according to the output indices, we obtain 
\begin{align*}
\widetilde{\mathscr{C}}(j_1,j_2,j_3)
&= 
\sum_{\alpha_1=1}^{r_1}
\sum_{\alpha_2=1}^{r_2}
\sum_{\beta_1=1}^{s_1}
\sum_{\beta_2=1}^{s_2}
\mathscr{G}_1(1,j_1,\alpha_1)
\Bigg(
\sum_{j=1}^{l}
\mathscr{G}_2(\alpha_1,j,\alpha_2)
\mathscr{H}_1(1,j,\beta_1)\\
&\times\mathscr{H}_2(\beta_1,j_2,\beta_2)
\Bigg)
\widetilde{\mathscr{G}}_3(\alpha_2,j_3,1)\,
\widetilde{\mathscr{H}}_3(\beta_2,j_3,1).
\end{align*}
Introducing the combined indices $\gamma_1=\phi(\alpha_1,\beta_1)$ and $\gamma_2=\phi(\alpha_2,\beta_2)$ and using the definitions of $\mathscr{P}_1$, $\mathscr{P}_2$, and $\widetilde{\mathscr{P}}_3$, we obtain
\begin{equation*}
    \widetilde{\mathscr{C}}(j_1,j_2,j_3)
=
\sum_{\gamma_1=1}^{r_1s_1}
\sum_{\gamma_2=1}^{r_2s_2}
\mathscr{P}_1(1,j_1,\gamma_1)\,
\mathscr{P}_2(\gamma_1,j_2,\gamma_2)\,
\widetilde{\mathscr{P}}_3(\gamma_2,j_3,1).
\end{equation*}
Applying the inverse Fourier transform to the third core $\mathscr{P}_3 = \widetilde{\mathscr{P}}_3 \times_2 \textbf{F}_{n_3}^{-1}$ yields the TTD representation of $\mathscr{C}$.
\end{proof}

\begin{proposition}[HTD/TD-based T-product]\label{prop:htdtd_tprod}
Given two third-order tensors 
$\mathscr{T}\in\mathbb{R}^{n_1\times l\times n_3}$ and
$\mathscr{S}\in\mathbb{R}^{l\times n_2\times n_3}$ in the HTD/TD format with core tensors $\mathscr{G}\in\mathbb{R}^{r_1\times r_2\times r_3}$ and $\mathscr{H}\in\mathbb{R}^{s_1\times s_2\times s_3}$ and factor matrices $\textbf{U}_t$ and $\textbf{V}_t$, $t=1,2,3$, respectively, their T-product $\mathscr{T}\star\mathscr{S}\in\mathbb{R}^{n_1\times n_2\times n_3}$ admits an HTD/TD representation with core tensor $\mathscr{P} = \widetilde{\mathscr{P}} \times_3\textbf{F}_{n_3}^{-1}\in\mathbb{R}^{r_1\times s_2\times n_3}$ where
\begin{equation*}
\scalebox{0.92}{$\displaystyle
\widetilde{\mathscr{P}}(\alpha_1,\beta_2,\gamma_3)
=
\sum_{\alpha_2=1}^{r_2}
\sum_{\beta_1=1}^{s_1}
\sum_{\alpha_3=1}^{r_3}
\sum_{\beta_3=1}^{s_3}
\mathscr{G}(\alpha_1,\alpha_2,\alpha_3)
\textbf{M}(\alpha_2,\beta_1)
\mathscr{H}(\beta_1,\beta_2,\beta_3) 
\widetilde{\textbf{U}}_3(\gamma_3,\alpha_3)
\widetilde{\textbf{V}}_3(\gamma_3,\beta_3),$}
\end{equation*}
with $\textbf{M}=\textbf{U}_{2}^\top\textbf{V}_{1}
\in\mathbb{R}^{r_2\times s_1}$,  $\widetilde{\textbf{U}}_3=\textbf{F}_{n_3}\textbf{U}_3\in\mathbb{C}^{n_3\times r_3}$, and $\widetilde{\textbf{V}}_3=\textbf{F}_{n_3}\textbf{V}_3\in\mathbb{C}^{n_3\times s_3}$, and with factor matrices $\textbf{U}_1$, $\textbf{V}_2$, and $\textbf{I}_{n_3}$. 
\end{proposition}
\begin{proof}
The proof proceeds similarly to that for the TTD-based T-product.
    Denote by $\mathscr{C}=\mathscr{T}\star\mathscr{S}$  the T-product and by $\widetilde{\mathscr{C}}$  its representation in the Fourier domain. According to \eqref{eq:htdbd}, it follows that
    \begin{align*}
        \widetilde{\mathscr{C}}(j_1,j_2,j_3) &= \sum_{j=1}^{l}
\sum_{\alpha_1=1}^{r_1}
\sum_{\alpha_2=1}^{r_2}
\sum_{\alpha_3=1}^{r_3}
\sum_{\beta_1=1}^{s_1}
\sum_{\beta_2=1}^{s_2}
\sum_{\beta_3=1}^{s_3}
\mathscr{G}(\alpha_1,\alpha_2,\alpha_3)
\mathscr{H}(\beta_1,\beta_2,\beta_3)\\
&\times \textbf{U}_1(j_1,\alpha_1)
\textbf{U}_2(j,\alpha_2)
\textbf{V}_1(j,\beta_1)
\textbf{V}_2(j_2,\beta_2)\widetilde{\textbf{U}}_3(j_3,\alpha_3)
\widetilde{\textbf{V}}_3(j_3,\beta_3)\\
& = \sum_{\alpha_1=1}^{r_1}
\sum_{\beta_2=1}^{s_2}\widetilde{\mathscr{P}}(\alpha_1,\beta_2,j_3)
\textbf{U}_{1}(j_1,\alpha_1)
\textbf{V}_{2}(j_2,\beta_2).
    \end{align*}
Finally, applying the inverse Fourier transform to the third mode of $\widetilde{\mathscr{P}}$ yields the HTD/TD representation of $\mathscr{C}$.
\end{proof}

\begin{remark}
Assume $n_1 = n_2 = n_3 = n$, and  let $r$ and $s$ denote the maximum TT-ranks or hierarchical ranks of $\mathscr{T}$ and $\mathscr{S}$. The computational complexities of the TTD-based and HTD/TD-based T-product are about  $\mathcal{O}((r+s+rs)n\log n
+
r^2sl
+
r^2s^2n
)$ and $\mathcal{O}(
(r+s+rs)n\log n + rsl +
(r^3 + s^3 + r^2s + rs^2)n)$, respectively. By comparison, the standard block-circulant implementation of the T-product requires approximately 
$\mathcal{O}(l n^4)$ operations and the explicit storage of the full block-circulant matrix. 
\end{remark}


The TTD-based and HTD/TD-based formulations compute the T-product entirely in compressed form, avoiding full tensors and their block-circulant representations and thereby reducing computational cost and memory requirements when the decomposition ranks are small relative to the tensor dimensions.

\subsection{T-SVD} Similar to \cref{coro:ttdtproduct} for the T-product, we can  leverage tensor decomposition-based block diagonalization to compute T-SVD. 

\begin{corollary}
    Given a third-order tensor $\mathscr{T}\in\mathbb{R}^{n_1\times n_2\times n_3}$ in the TTD format with core tensors $\mathscr{G}_1$, $\mathscr{G}_2$, and $\mathscr{G}_3$, let $\textbf{T}_{j_3}=\textbf{U}_{j_3}\textbf{S}_{j_3}\textbf{V}_{j_3}^*$ be the matrix SVDs of the diagonal blocks $\textbf{T}_1,\textbf{T}_2,\dots,\textbf{T}_{n_3}$ computed from the TTD representation according to \eqref{eq:ttdbd}. Then the T-SVD of $\mathscr{T}$ is obtained by assembling these matrix SVD factors and applying the inverse Fourier transform.
\end{corollary}
\begin{proof}
    The result follows directly from the definition of T-SVD and the block diagonalization established in \cref{prop:ttdbd}.
\end{proof}

Analogously,  T-SVD can also be computed using the block-diagonal representation associated with HTD/TD as presented in \cref{prop:htdbd}. Furthermore, the factor tensors of the T-SVD can be constructed directly within the TTD or HTD/TD formats by operating on the corresponding core representations, thereby avoiding the explicit formation of the full dense tensors in either the original or transform domains. For convenience, we permit the TTD and HTD/TD representations of the factor tensors of T-SVD to be complex-valued, even though equivalent real-valued representations can be readily obtained via suitable transformations.

\begin{proposition} [TTD-based T-SVD]
\label{prop:ttd_tsvd}
    For a third-order tensor $\mathscr{T}\in\mathbb{R}^{n_1\times n_2\times n_3}$ in the TTD format with core tensors $\mathscr{G}_1$, $\mathscr{G}_2$, and $\mathscr{G}_3$ and TT-ranks $\{1,r_1,r_2,1\}$, let $\textbf{G}_1=\textbf{Q}_1\textbf{R}_1$ and $\textbf{G}_2=\textbf{Q}_2\textbf{R}_2$ be the QR factorizations of the matricization $\textbf{G}_1\in\mathbb{R}^{n_1\times r_1}$ of $\mathscr{G}_1$ and the mode-2 matricization $\textbf{G}_2\in\mathbb{R}^{n_2\times r_1r_2}$ of $\mathscr{G}_2$, respectively. Define $\textbf{M}_{\beta}=\textbf{R}_1\textbf{B}_{\beta}$ with
\begin{align*}
\textbf{B}_{\beta}(\alpha_1,\eta)
=
\sum_{\alpha_2=1}^{r_2}
\textbf{R}_2(\eta, \phi(\alpha_1, \alpha_2))
\widetilde{\mathscr{G}}_3(\alpha_2,\beta,1),
\end{align*} where $\phi(\alpha_1, \alpha_2) = \alpha_1+(\alpha_2-1)r_1$ and $\widetilde{\mathscr{G}}_3=\mathscr{G}_3\times_2\textbf{F}_{n_3}$.
Then the TTDs of the T-SVD factor tensors $\mathscr{U}$, $\mathscr{S}$, and $\mathscr{V}$ are computed as
\begin{align*}
    \mathscr{U}&: \ 
    \mathscr{P}_1(1,j_1,\alpha)=\textbf{Q}_1(j_1,\alpha), \quad
    \mathscr{P}_2(\alpha,\xi,\beta)=\textbf{U}_{\beta}(\alpha,\xi), \quad
    \mathscr{P}_3(\beta,j_3,1)=\textbf{F}_{n_3}^{-1}(j_3,\beta),\\
    \mathscr{S}&: \ 
    \mathscr{Q}_1(1,\xi,\mu)=\textbf{I}_{s}(\xi,\mu), \quad
    \mathscr{Q}_2(\mu,\zeta,\beta)=\textbf{S}_{\beta}(\mu,\zeta), \quad
    \mathscr{Q}_3(\beta,j_3,1)=\textbf{F}_{n_3}^{-1}(j_3,\beta),\\
    \mathscr{V}&: \ 
    \mathscr{R}_1(1,j_2,\eta)=\textbf{Q}_2(j_2,\eta), \quad
    \mathscr{R}_2(\eta,\zeta,\beta)=\textbf{V}_{\beta}(\eta,\zeta), \quad
    \mathscr{R}_3(\beta,j_3,1)=\textbf{F}_{n_3}^{-1}(j_3,\beta),
\end{align*}
where $\textbf{U}_\beta$, $\textbf{S}_\beta$, and $\textbf{V}_\beta$ are derived from the compact matrix SVD of $\textbf{M}_\beta$, i.e., $\textbf{M}_\beta=\textbf{U}_\beta\textbf{S}_\beta\textbf{V}_\beta^{*}$, and $s=\max_{1\leq \beta\leq n_3}\text{rank}(\textbf{M}_\beta)$, for $\alpha=1, 2, \dots,r_1$, $\beta=1,2,\dots,n_3$,
$\xi,\zeta,\mu=1,2,\dots,s$, and $\eta=1,2, \dots, \min\{n_2,r_1r_2\}$. If some $\textbf{M}_\beta$ has rank smaller than $s$, we extend its SVD factors to size $s$ by adding extra columns and setting the corresponding singular values to zero.
\end{proposition}

\begin{proof}
    According to \cref{prop:ttdbd}, the $\beta$th Fourier block diagonal matrix of $\mathscr{T}$ can be written as $\textbf{T}_\beta(j_1,j_2)
=
\sum_{\alpha_1=1}^{r_1}
\sum_{\alpha_2=1}^{r_2}
\textbf{G}_1(j_1,\alpha_1)
\mathscr{G}_2(\alpha_1,j_2,\alpha_2)
\widetilde{\mathscr{G}}_3(\alpha_2,\beta,1).$
The QR factorization of $\textbf{G}_2$ gives $\mathscr{G}_2(\alpha_1,j_2,\alpha_2)
=
\sum_{\eta=1}^{q}
\textbf{Q}_2(j_2,\eta)
\textbf{R}_2(\eta,\alpha_1+(\alpha_2-1)r_1).$
By the definition of $\textbf{B}_\beta$, it follows that $\textbf{T}_\beta
=
\textbf{G}_1\textbf{B}_\beta\textbf{Q}_2^\top.$
Applying the QR factorization $\textbf{G}_1=\textbf{Q}_1\textbf{R}_1$ yields
\begin{align*}
\textbf{T}_\beta
=
\textbf{Q}_1(\textbf{R}_1\textbf{B}_\beta)\textbf{Q}_2^\top
=
\textbf{Q}_1\textbf{M}_\beta\textbf{Q}_2^\top.
\end{align*}
Taking the compact SVD of the reduced matrix $\textbf{M}_\beta
=
\textbf{U}_\beta
\textbf{S}_\beta
\textbf{V}_\beta^{*}$ yields 
\begin{align*}
\textbf{T}_\beta
=
(\textbf{Q}_1\textbf{U}_\beta)
\textbf{S}_\beta
(\textbf{Q}_2\textbf{V}_\beta)^{*}.
\end{align*}
Since $\textbf{Q}_1$, $\textbf{Q}_2$, $\textbf{U}_\beta$, and $\textbf{V}_\beta$ have orthonormal columns, $\textbf{Q}_1\textbf{U}_\beta$ and $\textbf{Q}_2\textbf{V}_\beta$ also have orthonormal columns. Therefore, this yields a valid compact SVD of the Fourier block diagonal matrix $\textbf{T}_\beta$. Stacking these factorizations over $\beta=1,2,\dots,n_3$ yields the Fourier-domain tensors whose slices are $(\textbf{Q}_1\textbf{U}_\beta)$, $\textbf{S}_\beta$, and $(\textbf{Q}_2\textbf{V}_\beta)$. Applying the inverse Fourier transform along the third mode reconstructs the corresponding TT core tensors. 
For example, for $\mathscr{U}$, we obtain
\begin{align*}
    \mathscr{U}(j_1,\xi,j_3) = \sum_{\beta=1}^{n_3}
(\textbf{Q}_1\textbf{U}_{\beta})(j_1,\xi)
\textbf{F}_{n_3}^{-1}(j_3,\beta).
\end{align*}
This is equivalently written in TTD form as
\begin{equation*}
    \mathscr{U}(j_1,\xi,j_3) = \sum_{\beta=1}^{n_3}
\sum_{\alpha=1}^{r_1}
\mathscr{P}_1(1,j_1,\alpha)
\mathscr{P}_2(\alpha,\xi,\beta)
\mathscr{P}_3(\beta,j_3,1),
\end{equation*}
with TT-ranks $\{1,r_1,n_3,1\}$.
The constructions for $\mathscr{S}$ and $\mathscr{V}$ follow identically by replacing $\textbf{Q}_1\textbf{U}_\beta$ with $\textbf{S}_\beta$ and $\textbf{Q}_2\textbf{V}_\beta$, respectively.
\end{proof}

\begin{proposition} [HTD/TD-based T-SVD]
\label{prop:htd_tsvd}
Let $\mathscr{T}\in\mathbb{R}^{n_1\times n_2\times n_3}$ be a third-order
tensor in the HTD/TD format with core tensor
    $\mathscr{G}\in\mathbb{R}^{r_1\times r_2\times r_3}$ and factor matrices $\textbf{U}_t\in\mathbb{R}^{n_t\times r_t}$, $t= 1,2,3$. Assume $\textbf{U}_1$ and $\textbf{U}_2$ have orthonormal columns. Define $\widetilde{\textbf{U}}_3
=
\textbf{F}_{n_3}\textbf{U}_3$ and 
\begin{equation*}
    \textbf{G}_{j_3}
=
\sum_{\alpha_3=1}^{r_3}
\mathscr{G}(:,:, \alpha_3)
\widetilde{\textbf{U}}_3(j_3,\alpha_3)
\in \mathbb{C}^{r_1\times r_2}
\end{equation*}
for $j_3=1,2,\dots,n_3$. Suppose that the compact matrix SVDs of $\textbf{G}_{j_3}$ are given by $\textbf{G}_{j_3}=\textbf{P}_{j_3}\textbf{Q}_{j_3}\textbf{R}_{j_3}^{*}$ with $s=\max_{1\leq j_3\leq n_3}\text{rank}(\textbf{G}_{j_3})$. Then the HTD/TD representations of the T-SVD factor tensors $\mathscr{U}$, $\mathscr{S}$, and $\mathscr{V}$ are computed as
\begin{align*}
    \mathscr{U} &= \widetilde{\mathscr{P}}\times_1\textbf{U}_1\times_2\textbf{I}_{s}\times_3\textbf{F}_{n_3}^{-1}, \ \mathscr{S} = \widetilde{\mathscr{Q}}\times_1\textbf{I}_s\times_2\textbf{I}_s\times_3\textbf{F}_{n_3}^{-1},\\ 
    \mathscr{V} &= \widetilde{\mathscr{R}}\times_1\textbf{U}_2\times_2\textbf{I}_s\times_3\textbf{F}_{n_3}^{-1},
\end{align*}
where $\widetilde{\mathscr{P}}(:,:,j_3)=\textbf{P}_{j_3}$, $\widetilde{\mathscr{Q}}(:,:,j_3)=\textbf{Q}_{j_3}$, and $\widetilde{\mathscr{R}}(:,:,j_3)=\textbf{R}_{j_3}$. If some $\textbf{G}_{j_3}$ has rank smaller than $s$, we extend its SVD
factors to size $s$ by adding extra columns and setting the corresponding
singular values to zero.
\end{proposition}
\begin{proof}
    According to \cref{prop:htdbd}, the $j_3$th Fourier block diagonal matrix of $\mathscr{T}$ can be written as $\textbf{T}_{j_3}=\textbf{U}_1\textbf{G}_{j_3}\textbf{U}_2^\top$.
    Substituting the compact SVD of $\textbf{G}_{j_3}$ yields
    \begin{equation*}
        \textbf{T}_{j_3}
=
\textbf{U}_1
\textbf{P}_{j_3}
\textbf{Q}_{j_3}
\textbf{R}_{j_3}^*
\textbf{U}_2^\top
=
(\textbf{U}_1\textbf{P}_{j_3})
\textbf{Q}_{j_3}
(\textbf{U}_2\textbf{R}_{j_3})^*.
    \end{equation*}
Since $\textbf{U}_1$ and $\textbf{U}_2$ have orthonormal columns, the factor matrices $\textbf{U}_1\textbf{P}_{j_3}$ and $\textbf{U}_2\textbf{R}_{j_3}$ have orthonormal columns. Therefore, this defines a valid matrix SVD of $\textbf{T}_{j_3}$. By the definition of the T-SVD, the Fourier-domain frontal slices of the factor tensors $\mathscr{U}$, $\mathscr{S}$, and $\mathscr{V}$ are respectively given by $\textbf{U}_1\textbf{P}_{j_3}$, $\textbf{Q}_{j_3}$, and $\textbf{U}_2\textbf{R}_{j_3}$ for $j_3=1,2,\dots,n_3$. Stacking these slices yields the Fourier-domain tensors $\widetilde{\mathscr{P}}\times_1\textbf{U}_1$, $\widetilde{\mathscr{Q}}$, and $\widetilde{\mathscr{R}}\times_1\textbf{U}_2$. Applying the inverse Fourier transform along the third mode gives  the HTD/TD representations of the T-SVD factor tensors. 
\end{proof}

The factor tensors/matrices in \cref{prop:ttd_tsvd} and \cref{prop:htd_tsvd} may be complex-valued, but equivalent real-valued representations can be obtained by grouping conjugate frequency pairs. The Fourier-domain frontal slices of $\mathscr{T}$ satisfy $\textbf{T}_{\bar{j}_3}=\overline{\textbf{T}}_{j_3}$, where $\bar{j}_3=1$ for $j_3=1$ and
$\bar{j}_3=n_3-j_3+2$ otherwise. Thus, the left and right singular factors at $\bar{j}_3$ can be chosen as the complex conjugates of those at $j_3$ with the same singular values. In the inverse Fourier transform, each conjugate pair can be transformed into real and imaginary linear combinations, yielding real-valued cores without changing the T-SVD factor tensors or the stated ranks. The self-conjugate frequencies $j_3=1$ and, when $n_3$ is even, $j_3=n_3/2+1$, are real-valued by the same symmetry relation. Therefore, the T-SVD factor tensors admit equivalent real-valued TTD or HTD/TD representations after grouping all conjugate frequency pairs.

\begin{remark}
Assume $n_1 = n_2 = n_3 = n$, and let $r$ denote the maximum TT-rank or hierarchical rank. For the TTD-based T-SVD, the main computations consist of three parts, namely transforming the third core tensor, computing the QR factorization of the reshaped first core tensor,
and performing compact SVDs of $\textbf{M}_{j_3}$ for $j_3=1,2,\dots,n_3$. The resulting cost is thus approximately $\mathcal{O}(rn\log n+nr^2+ nr^2q + n\min \{r, q\}^2 \max \{r, q\})$, where $q \leq \min\{r^2, n_2\}$.
For the HTD/TD-based T-SVD, the dominant computations consist of transforming $\mathbf{U}_3$, constructing the reduced matrices $\mathbf{G}_{j_3}$, and performing compact SVDs of $\mathbf{G}_{j_3}$ for $j_3=1,2,\ldots,n_3$, yielding an overall complexity of $\mathcal{O}(rn\log n+nr^3).$ By comparison, the computational cost of the definition-based T-SVD is dominated by the Fourier transform of $\mathscr{T}$ and the matrix SVDs of the Fourier-domain frontal slices, which cost approximately $\mathcal{O}(n^3\log n+n^4)$.
\end{remark}


In summary, the proposed TTD- and HTD/TD-based formulations compute the T-SVD directly from the corresponding decomposition factors and obtain the resulting factor tensors in the TTD or HTD/TD format, without explicitly reconstructing either the full tensor or its block-diagonal representation. When the TT-ranks or hierarchical ranks are small relative to the tensor dimensions, the proposed formulations achieve substantial computational savings over the definition-based T-SVD.


\section{Generalization to higher-order tensors} 
\label{sec:higherorderoperation}
This section generalizes the preceding results to the higher-order setting.  For a $k$th-order tensor
$\mathscr{T}\in\mathbb{R}^{n_1\times n_2\times\cdots\times n_k}$, the
definition-based higher-order block diagonalization first applies separable
discrete Fourier transforms along modes $3,4,\dots,k$, i.e., 
\begin{align*}
\widetilde{\mathscr{T}}
=
\mathscr{T}
\times_3\textbf{F}_{n_3}
\times_4\textbf{F}_{n_4} \times_5
\cdots
\times_k\textbf{F}_{n_k}.
\end{align*}
For each frequency tuple
$\boldsymbol{j}=(j_3,j_4,\dots,j_k)$, where
$j_t=1,2,\dots,n_t$ for $t=3,4,\dots,k$, the corresponding Fourier-domain
block is $\textbf{T}_{\boldsymbol{j}}
=
\widetilde{\mathscr{T}}(:,:,j_3,j_4,\dots,j_k)
\in\mathbb{C}^{n_1\times n_2}.$
Instead of explicitly constructing the full tensor and its Fourier transform, the following propositions derive  higher-order block diagonalization, the higher-order tensor-tensor product, and higher-order SVD directly from the TTD and HTD representations by operating on their core tensors, transfer tensors, or factor matrices.


\begin{proposition}[TTD-based higher-order block diagonalization]
\label{thm:ttd_kth_block_diag}
Let $\mathscr{T}\in\mathbb{R}^{n_1\times n_2\times n_3\times\cdots\times n_k}$ be a tensor in the TTD format with core tensors
$\mathscr{G}_t\in\mathbb{R}^{r_{t-1}\times n_t\times r_t}$,
$t=1,2, \dots,k$. Define $\widetilde{\mathscr{G}}_t
=
\mathscr{G}_t\times_2 \textbf{F}_{n_t}
\in
\mathbb{C}^{r_{t-1}\times n_t\times r_t}$ for 
$t=3, 4, \dots,k$.
For each frequency tuple
$\boldsymbol{j}=(j_3,j_4, \dots,j_k)$, the Fourier-domain diagonal block
$\textbf{T}_{\boldsymbol{j}}\in\mathbb{C}^{n_1\times n_2}$ is then computed as
\begin{align*}
\textbf{T}_{\boldsymbol{j}}(j_1,j_2)
=
\sum_{\alpha_1=1}^{r_1}\sum_{\alpha_2=1}^{r_2}
\cdots
\sum_{\alpha_{k}=1}^{r_{k}}
\mathscr{G}_1(1,j_1,\alpha_1)
\mathscr{G}_2(\alpha_1,j_2,\alpha_2)
\prod_{t=3}^{k}
\widetilde{\mathscr{G}}_t(\alpha_{t-1},j_t,\alpha_t).
\end{align*}
\end{proposition}

\begin{proof}
The proof follows the same argument as \cref{prop:ttdbd}. The higher-order block diagonalization applies the discrete Fourier transforms along modes $3, 4, \dots,k$. Since mode-$t$ multiplication acts only on the physical index of the $t$th  core tensor, the transformed tensor is represented by the original core tensors $\mathscr{G}_1,\mathscr{G}_2$ and the transformed core tensors $\widetilde{\mathscr{G}}_t$ for $t=3, 4, \dots,k$. Evaluating this TTD representation at the frequency tuple $\boldsymbol{j}$ gives the stated expression for $\textbf{T}_{\boldsymbol{j}}$.
\end{proof}



\begin{proposition}[HTD-based higher-order block diagonalization]
\label{thm:htd_kth_block_diag}
Let $\mathscr{T}\in\mathbb{R}^{n_1\times n_2\times n_3\times\cdots\times n_k}$
be a tensor in the HTD format with respect to a dimension tree $\mathcal{T}$ on
$D=\{1,2, \dots,k\}$ with $\{1,2\}\in\mathcal{T}$. Let
$\textbf{U}_t\in\mathbb{R}^{n_t\times r_t}$ denote the leaf factor matrices and
 $\mathscr{B}_\tau$  the transfer tensor associated with each internal
node $\tau\in\mathcal{T}$. Define $\widetilde{\textbf{U}}_t
=
\textbf{F}_{n_t}\textbf{U}_t
\in
\mathbb{C}^{n_t\times r_t}$ for $t=3, 4, \dots,k$.
For a fixed frequency tuple
$\boldsymbol{j}=(j_3,j_4, \dots,j_k)$, evaluate the transformed leaf
factor $\widetilde{\textbf{U}}_t$ at the row
$\widetilde{\textbf{U}}_t(j_t,:)$, while leaving
$\textbf{U}_1$ and $\textbf{U}_2$ unevaluated. Contracting all transfer tensors
and the selected leaf rows over the dimension tree, with the indices associated
with modes $1$ and $2$ left open, gives a reduced matrix $\textbf{G}_{\boldsymbol{j}}
\in
\mathbb{C}^{r_1\times r_2}.$ 
Then the Fourier-domain diagonal block corresponding to each frequency tuple $\boldsymbol{j}$ is computed as
\begin{align*}
\textbf{T}_{\boldsymbol{j}}
=
\textbf{U}_1
\textbf{G}_{\boldsymbol{j}}
\textbf{U}_2^{\top}
\in
\mathbb{C}^{n_1\times n_2}.
\end{align*}
\end{proposition}

\begin{proof}
The argument is analogous to the third-order HTD construction in \cref{prop:htdbd}, but the remaining
contractions are carried out along the dimension tree. Applying the discrete
Fourier transforms along modes $3, 4, \dots,k$ is equivalent to replacing the leaf
factor matrices $\textbf{U}_t$ by
$\widetilde{\textbf{U}}_t=\textbf{F}_{n_t}\textbf{U}_t$, $t=3, 4, \dots,k$.
For a fixed frequency tuple $\boldsymbol{j}$, selecting the rows
$\widetilde{\textbf{U}}_t(j_t,:)$ fixes all transformed modes. Contracting
the remaining HTD network gives
$\textbf{T}_{\boldsymbol{j}}
=\textbf{U}_1\textbf{G}_{\boldsymbol{j}}\textbf{U}_2^\top$.
\end{proof}


\begin{proposition}[TTD-based higher-order tensor-tensor product]
\label{thm:ttd_kth_tprod_core} Given two tensors $\mathscr{T}\in\mathbb{R}^{n_1\times l\times n_3\times\cdots\times n_k}$
and
$\mathscr{S}\in\mathbb{R}^{l\times n_2\times n_3\times\cdots\times n_k}$
in TTD format with core tensors $\mathscr{G}_t, \ \mathscr{H}_t$ and TT-ranks $\{1,r_1,\ldots,r_{k-1},1\}$, $\{1,s_1,\ldots,s_{k-1},1\}$,
respectively, define $\widetilde{\mathscr{G}}_t
=
\mathscr{G}_t\times_2\textbf{F}_{n_t}$ and $
\widetilde{\mathscr{H}}_t
=
\mathscr{H}_t\times_2\textbf{F}_{n_t}$ for 
$t=3, 4, \dots,k$.
Then the higher-order tensor-tensor product $\mathscr{T}\star_H\mathscr{S}$ admits a TTD representation
with core tensors 
\begin{align*}
\begin{cases}
    &\mathscr{P}_1(1,j_1,\phi(\alpha_1,\beta_1))
=
\mathscr{G}_1(1,j_1,\alpha_1),\\
& \mathscr{P}_2(\phi(\alpha_1,\beta_1),j_2,\phi(\alpha_2,\beta_2))
=
\displaystyle\sum_{j=1}^{l}
\mathscr{G}_2(\alpha_1,j,\alpha_2)
\mathscr{H}_1(1,j,\beta_1)
\mathscr{H}_2(\beta_1,j_2,\beta_2),\\
& \mathscr{P}_t
=
\widetilde{\mathscr{P}}_t\times_2\textbf{F}_{n_t}^{-1} \text{ for } t = 3,4, \cdots, k,
\end{cases}
\end{align*}
where $\phi(\alpha_t,\beta_t)=\alpha_t+(\beta_t-1)r_t$, and for $\alpha_t=1,2,\dots,r_t$ and $\beta_t=1,2,\dots,s_t$,
\begin{equation*}
    \widetilde{\mathscr{P}}_t
\bigl(\phi(\alpha_{t-1},\beta_{t-1}),j_t,
\phi(\alpha_t,\beta_t)\bigr)
=
\widetilde{\mathscr{G}}_t(\alpha_{t-1},j_t,\alpha_t)
\widetilde{\mathscr{H}}_t(\beta_{t-1},j_t,\beta_t),
\end{equation*}
and with TT-ranks $\{1,r_1s_1,r_2s_2,\dots,r_{k-1}s_{k-1},1\}$.

\end{proposition}

\begin{proof}
The proof follows the same argument as \cref{prop:ttd_tprod}. After applying the discrete Fourier transform along modes $3, 4, \dots,k$, the higher-order tensor-tensor product reduces to matrix multiplication at each frequency tuple
$\boldsymbol{j}$. By \cref{thm:ttd_kth_block_diag}, the
Fourier-domain blocks of $\mathscr{T}$ and $\mathscr{S}$ are obtained from the
core tensors $\mathscr{G}_1,\mathscr{G}_2,\widetilde{\mathscr{G}}_t$ and
$\mathscr{H}_1,\mathscr{H}_2,\widetilde{\mathscr{H}}_t$, respectively.
Multiplying these two blocks and contracting over the shared index
$j=1,2, \dots,l$ gives the first two core tensors $\mathscr{P}_1$ and $\mathscr{P}_2$,
while the remaining frequency-mode core tensors are obtained by pairing the
corresponding entries of $\widetilde{\mathscr{G}}_t$ and
$\widetilde{\mathscr{H}}_t$, $t=3, 4, \dots,k$. Applying $\textbf{F}_{n_t}^{-1}$ to
these paired core tensors along their second indices gives the stated  core tensors.
\end{proof}


\begin{proposition}[HTD-based higher-order tensor-tensor product]
\label{thm:htd_kth_tprod_core}
Given two tensors
$\mathscr{T}\in\mathbb{R}^{n_1\times l\times n_3\times\cdots\times n_k}$
and
$\mathscr{S}\in\mathbb{R}^{l\times n_2\times n_3\times\cdots\times n_k}$
in the HTD format with respect to the same dimension tree
$\mathcal{T}$ on $D=\{1,2, \dots,k\}$ with $\{1,2\}\in\mathcal{T}$,
let $\textbf{U}_t$ and $\mathscr{G}_\tau$ and $\textbf{V}_t$ and $\mathscr{H}_\tau$
denote the leaf factor matrices and transfer tensors of $\mathscr{T}$ and
$\mathscr{S}$, respectively. Define $\phi(a,b)=a+(b-1)p$, where $a = 1,2, \dots,p$.
Let $\textbf{W}_1=\textbf{U}_1$, $\textbf{W}_2=\textbf{V}_2$, and define
$\widetilde{\textbf{U}}_t=\textbf{F}_{n_t}\textbf{U}_t$,
$\widetilde{\textbf{V}}_t=\textbf{F}_{n_t}\textbf{V}_t$,
$\widetilde{\textbf{W}}_t(j_t,\phi(\alpha_t,\beta_t))
=
\widetilde{\textbf{U}}_t(j_t,\alpha_t)
\widetilde{\textbf{V}}_t(j_t,\beta_t),
\ \text{and }
\textbf{W}_t=\textbf{F}_{n_t}^{-1}\widetilde{\textbf{W}}_t$ for
$t=3, 4, \dots,k$. At the node
$\{1,2\}$, let $\textbf{M}=\textbf{U}_2^\top\textbf{V}_1\in\mathbb{R}^{r_2\times s_1}$
and define
\begin{align*}
\mathscr{P}_{\{1,2\}}(\alpha_1,\beta_2,\phi(\gamma,\delta))
=
\sum_{\alpha_2=1}^{r_2}\sum_{\beta_1=1}^{s_1}
\mathscr{G}_{\{1,2\}}(\alpha_1,\alpha_2,\gamma)
\textbf{M}(\alpha_2,\beta_1)
\mathscr{H}_{\{1,2\}}(\beta_1,\beta_2,\delta).
\end{align*}
For any other internal node $\tau\in\mathcal{T}$ with children $\tau_1$ and
$\tau_2$, define
\begin{align*}
\mathscr{P}_\tau
\bigl(\phi(\alpha_{\tau_1},\beta_{\tau_1}),
      \phi(\alpha_{\tau_2},\beta_{\tau_2}),
      \phi(\alpha_\tau,\beta_\tau)\bigr)
=
\mathscr{G}_\tau(\alpha_{\tau_1},\alpha_{\tau_2},\alpha_\tau)
\mathscr{H}_\tau(\beta_{\tau_1},\beta_{\tau_2},\beta_\tau).
\end{align*}
Then the higher-order tensor-tensor product $\mathscr{T}\star_H\mathscr{S}$ admits an HTD representation with leaf factors $\textbf{W}_t$ and transfer tensors $\mathscr{P}_\tau$, whose hierarchical ranks
are bounded by the products of the corresponding input ranks.
\end{proposition}

\begin{proof}
The argument is similar to the third-order HTD construction in \cref{prop:htdtd_tprod}. After applying the discrete Fourier transform along modes $3, 4, \dots,k$, the tensor-tensor
product reduces to matrix multiplication at each frequency tuple. The transformed
leaf factors $\widetilde{\textbf{U}}_t$ and $\widetilde{\textbf{V}}_t$,
$t=3, 4, \dots,k$, therefore give the paired output leaf factors
$\widetilde{\textbf{W}}_t$. The shared dimension of size $l$ is contracted at the
node $\{1,2\}$ through $\textbf{M}=\textbf{U}_2^\top\textbf{V}_1$, which gives
the stated transfer tensor $\mathscr{P}_{\{1,2\}}$. All other internal nodes are
obtained by pairing the corresponding transfer tensors of the two input HTD
representations. Applying the inverse discrete Fourier transform to the paired leaf factors along modes
$3, 4, \dots,k$ gives the stated HTD representation.
\end{proof}


\begin{proposition}[TTD-based higher-order SVD]
\label{thm:ttd_kth_tsvd_core}
Let $\mathscr{T}\in\mathbb{R}^{n_1\times n_2\times\cdots\times n_k}$ be a tensor in the TTD format with core tensors $\mathscr{G}_t\in\mathbb{R}^{r_{t-1}\times n_t\times r_t}$, $t=1,2, \dots,k$. Define $\widetilde{\mathscr{G}}_t=\mathscr{G}_t\times_2\textbf{F}_{n_t}$ for $t=3,4,\dots,k$, and let $\textbf{G}_1=\textbf{Q}_1\textbf{R}_1$ and $\textbf{G}_2=\textbf{Q}_2\textbf{R}_2$ be the QR factorizations of the matricization $\textbf{G}_1\in\mathbb{R}^{n_1\times r_1}$ of $\mathscr{G}_1$ and the mode-2 matricization $\textbf{G}_2\in\mathbb{R}^{n_2\times r_1r_2}$ of $\mathscr{G}_2$, respectively. For each frequency tuple $\boldsymbol{\beta}=(\beta_3,\beta_4,\dots,\beta_k)$, where $\beta_t=1,2,\dots,n_t$,
\begin{align*}
\textbf{B}_{\boldsymbol{\beta}}(\alpha_1,\eta)
= \sum_{\alpha_2=1}^{r_2}\sum_{\alpha_3 = 1}^{r_3}
\cdots
\sum_{\alpha_k=1}^{r_k}
\textbf{R}_2(\eta,\alpha_1+(\alpha_2-1)r_1)
\prod_{t=3}^{k}
\widetilde{\mathscr{G}}_t(\alpha_{t-1},\beta_t,\alpha_t).
\end{align*}
Define $\textbf{M}_{\boldsymbol{\beta}}
= \textbf{R}_1\textbf{B}_{\boldsymbol{\beta}}$, $s=\max_{\boldsymbol{\beta}}\operatorname{rank}
(\textbf{M}_{\boldsymbol{\beta}}),$ and
let $\textbf{M}_{\boldsymbol{\beta}}
=
\textbf{U}_{\boldsymbol{\beta}}
\boldsymbol{S}_{\boldsymbol{\beta}}
(\textbf{V}_{\boldsymbol{\beta}})^*$
be the compact SVD padded with zero singular values if necessary. Let $\rho_t= \phi(\beta_t, \rho_{t+1}) = \beta_t+(\rho_{t+1}-1)n_t$ denote the combined frequency index and set $\rho_{k+1}=1$.
Then the TTDs of the higher-order SVD factor tensors
$\mathscr{U}$, $\mathscr{S}$, and $\mathscr{V}$ are computed as
\begin{align*}
&\mathscr{U}:\mathscr{P}_1(1,j_1,\alpha)
=
\textbf{Q}_1(j_1,\alpha),
\
\mathscr{P}_2(\alpha,\xi,\rho_3)
=
\textbf{U}_{\boldsymbol{\beta}}(\alpha,\xi),
\
\mathscr{P}_t(\rho_t, j_t, \rho_{t+1})=\textbf{F}_{n_t}^{-1}(j_t, \beta_t),\\
&\mathscr{S}: \mathscr{Q}_1(1,\xi,\mu)=\textbf{I}_s(\xi,\mu),
\
\mathscr{Q}_2(\mu,\zeta,\rho_3)
=
\boldsymbol{S}_{\boldsymbol{\beta}}(\mu,\zeta),
\
\mathscr{Q}_t(\rho_t, j_t, \rho_{t+1}) =\textbf{F}_{n_t}^{-1}(j_t, \beta_t), \\
&\mathscr{V}: \mathscr{R}_1(1,j_2,\eta)=\textbf{Q}_2(j_2,\eta),
\
\mathscr{R}_2(\eta,\zeta,\rho_3)
=
\textbf{V}_{\boldsymbol{\beta}}(\eta,\zeta),
\
\mathscr{R}_t(\rho_t, j_t, \rho_{t+1}) =\textbf{F}_{n_t}^{-1}(j_t, \beta_t),
\end{align*}
for $t = 3, 4, \dots, k$, where $\eta=1,2,\dots,\min\{n_2,r_1r_2\}$ and all unspecified entries of the transform-mode core tensors equal to zero.
\end{proposition}

\begin{proof}
The proof follows the same argument as \cref{prop:ttd_tsvd}.
For each frequency tuple $\boldsymbol{\beta}$, \cref{thm:ttd_kth_block_diag}
and the QR factorization of the mode-2 matricization of $\mathscr{G}_2$ give
$\widetilde{\mathscr{T}}(:,:, \boldsymbol{\beta})
=
\textbf{G}_1\textbf{B}_{\boldsymbol{\beta}}\textbf{Q}_2^\top.$
Using $\textbf{G}_1=\textbf{Q}_1\textbf{R}_1$, we have $\widetilde{\mathscr{T}}(:,:, \boldsymbol{\beta})
=
\textbf{Q}_1\textbf{M}_{\boldsymbol{\beta}}\textbf{Q}_2^\top.$
Taking the SVD of $\textbf{M}_{\boldsymbol{\beta}}$ gives $\widetilde{\mathscr{T}}(:,:, \boldsymbol{\beta})
=
(\textbf{Q}_1\textbf{U}_{\boldsymbol{\beta}})
\boldsymbol{S}_{\boldsymbol{\beta}}
(\textbf{Q}_2\textbf{V}_{\boldsymbol{\beta}})^*.$
Thus, the Fourier-domain higher-order SVD factors are obtained from
$\textbf{Q}_1\textbf{U}_{\boldsymbol{\beta}}$,
$\boldsymbol{S}_{\boldsymbol{\beta}}$, and
$\textbf{Q}_2\textbf{V}_{\boldsymbol{\beta}}$ for all frequency tuples. The stated TT core tensors collect these factors over the combined frequency indices $\rho_t$, while the core tensors $\mathscr{P}_t$, $\mathscr{Q}_t$, and $\mathscr{R}_t$, $t=3,4,\dots,k$, apply the inverse Fourier transforms along modes $3, 4, \dots,k$. Hence, they define TTD representations of the higher-order SVD factor tensors $\mathscr{U}$, $\mathscr{S}$, and $\mathscr{V}$.
\end{proof}

\begin{proposition}[HTD-based higher-order SVD]
\label{thm:htd_kth_tsvd_core}
Let $\mathscr{T}\in\mathbb{R}^{n_1\times n_2\times n_3\times\cdots\times n_k}$ be a tensor in the HTD format with leaf factor matrices $\textbf{U}_t$ and
transfer tensors $\mathscr{B}_\tau$ with respect to a dimension tree
$\mathcal{T}$ on $D=\{1,2, \dots,k\}$, where $\textbf{U}_1$ and
$\textbf{U}_2$ have orthonormal columns. Assume that $\{1,2\}\in\mathcal{T}$ and that the modes $3, 4, \dots,k$ form the complementary subtree. For $t=3, 4, \dots,k$, define
$\widetilde{\textbf{U}}_t=\textbf{F}_{n_t}\textbf{U}_t$. For each frequency tuple
$\boldsymbol{j}$, let
$\textbf{G}_{\boldsymbol{j}}\in\mathbb{C}^{r_1\times r_2}$ be obtained by
contracting the transformed HTD network with the rows
$\widetilde{\textbf{U}}_t(j_t,:)$, while leaving the rank
indices associated with modes $1$ and $2$ open. Set
$s=\max_{\boldsymbol{j}}\operatorname{rank}(\textbf{G}_{\boldsymbol{j}})$ and let $\textbf{G}_{\boldsymbol{j}}
=
\textbf{U}_{\boldsymbol{j}}
\boldsymbol{S}_{\boldsymbol{j}}
(\textbf{V}_{\boldsymbol{j}})^*$
be the compact SVD, padded with zero singular values if necessary.
Let $\mathcal{T}_{\mathscr{U}}$, $\mathcal{T}_{\mathscr{S}}$, and
$\mathcal{T}_{\mathscr{V}}$ be the HTD trees whose root $D$ has children $\{1,2\}$ and $f=\{3,4,\dots,k\}$, and which share the same subtree $\mathcal{T}_f$ on $f$. The leaf factors of $\mathscr{U}$, $\mathscr{S}$, and
$\mathscr{V}$ are computed as
\begin{align*}
\begin{array}{lll}
\textbf{L}^{\mathscr{U}}_{\{1\}}=\textbf{U}_1, &
\textbf{L}^{\mathscr{U}}_{\{2\}}=\textbf{I}_s, &
\textbf{L}^{\mathscr{U}}_{\{t\}}=\textbf{F}_{n_t}^{-1}, \\
\textbf{L}^{\mathscr{S}}_{\{1\}}=\textbf{I}_s, &
\textbf{L}^{\mathscr{S}}_{\{2\}}=\textbf{I}_s, &
\textbf{L}^{\mathscr{S}}_{\{t\}}=\textbf{F}_{n_t}^{-1}, \\
\textbf{L}^{\mathscr{V}}_{\{1\}}=\textbf{U}_2, &
\textbf{L}^{\mathscr{V}}_{\{2\}}=\textbf{I}_s, &
\textbf{L}^{\mathscr{V}}_{\{t\}}=\textbf{F}_{n_t}^{-1},
\end{array}
\qquad t=3, 4, \dots,k.
\end{align*}
Let $\phi(a,b)=a+(b-1)p$, $a = 1,2,\dots,p$. The transfer tensors at the node $\{1,2\}$ are
\begin{align*}
\mathscr{B}^{\mathscr{U}}_{\{1,2\}}(\alpha_1,\xi,\phi(\alpha_1,\xi))=1,\
\mathscr{B}^{\mathscr{S}}_{\{1,2\}}(\xi,\zeta,\phi(\xi,\zeta))=1,\
\mathscr{B}^{\mathscr{V}}_{\{1,2\}}(\alpha_2,\zeta,\phi(\alpha_2,\zeta))=1,
\end{align*}
For each leaf node $\{t\}\subseteq f$, let
$\rho_{\{t\}}=j_t$. For each internal node
$\tau\subseteq f$ with children $\tau_1$ and $\tau_2$, define $\rho_\tau = \phi(\rho_{\tau_1},\rho_{\tau_2}) = \rho_{\tau_1} + (\prod_{\ell\in\tau_1}n_\ell) (\rho_{\tau_2}-1).$
Therefore, 
the transfer tensors at $\tau$ satisfy
$
\mathscr{B}^{\mathscr{X}}_{\tau}
(\rho_{\tau_1},
\rho_{\tau_2},
\rho_\tau)=1$,
 where $\mathscr{X}\in\{\mathscr{U},\mathscr{S},\mathscr{V}\}$.
Finally, the 
root transfer tensors are computed as
\begin{align*}
&\mathscr{B}^{\mathscr{U}}_{D}(\phi(\alpha_1,\xi),\rho_f,1)
=
\textbf{U}_{\boldsymbol{j}}(\alpha_1,\xi),\quad
\mathscr{B}^{\mathscr{S}}_{D}(\phi(\xi,\zeta),\rho_f,1)
=
\boldsymbol{S}_{\boldsymbol{j}}(\xi,\zeta),\\
&\mathscr{B}^{\mathscr{V}}_{D}(\phi(\alpha_2,\zeta),\rho_f,1)
=
\textbf{V}_{\boldsymbol{j}}(\alpha_2,\zeta).
\end{align*}
These factor matrices and transfer tensors define HTD representations of the higher-order SVD factor tensors $\mathscr{U}$, $\mathscr{S}$, and $\mathscr{V}$, with all unspecified entries of the transfer tensors equal to zero.
\end{proposition}

\begin{proof}
The argument is similar to the third-order HTD construction in \cref{prop:htd_tsvd}. For each $\boldsymbol{j}$, \cref{thm:htd_kth_block_diag} gives $\widetilde{\mathscr{T}}(:,:, \boldsymbol{j})
=
\textbf{U}_1
\textbf{G}_{\boldsymbol{j}}
\textbf{U}_2^\top.$
Substituting the compact SVD of $\textbf{G}_{\boldsymbol{j}}$ yields $\widetilde{\mathscr{T}}(:,:, \boldsymbol{j})
=
(\textbf{U}_1\textbf{U}_{\boldsymbol{j}})
\boldsymbol{S}_{\boldsymbol{j}}
(\textbf{U}_2\textbf{V}_{\boldsymbol{j}})^* .$
Thus, the Fourier-domain higher-order SVD factors are
$\textbf{U}_1\textbf{U}_{\boldsymbol{j}}$,
$\boldsymbol{S}_{\boldsymbol{j}}$, and
$\textbf{U}_2\textbf{V}_{\boldsymbol{j}}$. The proposed HTD leaf factors
and transfer tensors reproduce these factors at each frequency tuple: the node
$\{1,2\}$ pairs the rank and singular-vector indices, the subtree $\mathcal{T}_f$ combines the frequency indices $j_3,j_4, \dots,j_k$ and applies the inverse discrete Fourier transforms through the leaf factors $\textbf{F}_{n_t}^{-1}$, and the root transfer tensor stores
$\textbf{U}_{\boldsymbol{j}}$,
$\boldsymbol{S}_{\boldsymbol{j}}$, or
$\textbf{V}_{\boldsymbol{j}}$. Since the transform-mode leaf factors are
$\textbf{F}_{n_t}^{-1}$, $t=3, 4, \dots,k$, the HTD contraction applies the inverse
discrete Fourier transforms along all transform modes. Therefore, the constructed HTD
representations give $\mathscr{U}$, $\mathscr{S}$, and $\mathscr{V}$ satisfying $\mathscr{T}=\mathscr{U}\star_H\mathscr{S}\star_H\mathscr{V}^{\top}.$
\end{proof}

\begin{remark}
Assume $n_1=n_2=\cdots=n_k=n$. Let $r$ and $s$ denote the maximum TT-ranks or hierarchical ranks. For the TTD-based framework, the computational complexities are $\mathcal{O}((k-2)r^2n\log n+n^{k-2}((k-2)r^2+nr^2 + rn^2))$ for higher-order block diagonalization, $\mathcal{O}(r^2sl + r^2s^2n+(k-2)n((r^2+s^2)\log n+r^2s^2(1+\log n)))$ for the higher-order tensor-tensor product, and
$\mathcal{O}((k-2)r^2n\log n + nr^2 + nr^2q + n^{k-2}((k-2)r^2 + r^2q + \min \{r, q\}^2\max\{r, q\}))$ where $q \leq \min\{r^2, n_2\}$ for higher-order SVD.
For the HTD-based framework, the corresponding computational complexities are $\mathcal{O}((k-2)nr\log n+n^{k-2}(kr^3+r^2n+rn^2))$ for higher-order block diagonalization, $\mathcal{O}((k-2)n(r+s+rs)\log n+lrs+kr^3s^3)$ for the higher-order tensor-tensor product, and $\mathcal{O}((k-2)nr\log n+n^{k-2}(kr^3+r^3))$ for  higher-order SVD.
In contrast, the definition-based approach explicitly forms the full tensors and performs computations on all $n^{k-2}$ block-diagonal matrices in the transform domain. Consequently, its computational complexities are $\mathcal{O}(n^k(k-2)\log n)$ for higher-order block diagonalization, $\mathcal{O}(n^{k-2}((2ln+n^2)(k-2)\log n+ln^2))$ for the higher-order tensor-tensor product, and $\mathcal{O}(n^k(k-2)\log n+n^{k-2}n^3)$ for  higher-order SVD.
\end{remark}


By formulating higher-order block diagonalization, the tensor-tensor product, and higher-order SVD directly in terms of the TTD and HTD factors, the proposed algorithms avoid constructing full tensors and their transform-domain representations while preserving the underlying low-rank decomposition structure. This reduces storage requirements and computational cost when the decomposition ranks are small relative to the tensor dimensions, providing an efficient framework for higher-order transform-based multilinear algebra.

\section{Numerical Examples} \label{sec:numeg}
We evaluated the proposed framework through a series of numerical experiments. All computations were performed on a university high-performance computing cluster, with each experiment allocated 200 GB of memory on a compute node. The source code for the experiments is available at \url{https://github.com/usernamemydusername/decomp_based_tensoralg}.

\subsection{Third-order T-product}

We evaluated the computational efficiency of the proposed TTD- and HTD-based T-product using randomly generated third-order tensors $\mathscr{T}\in\mathbb{R}^{n_1\times l\times n_3}$ and $\mathscr{S}\in\mathbb{R}^{l\times n_2\times n_3}$ under three data generation strategies: (a) random sparse tensors, (b) tensors with low TT-ranks, and (c) tensors with low hierarchical ranks. We excluded the computational time for constructing TTD and HTD formats. Throughout the experiments, we fixed $n_3=6$ and varied $n_1$, $n_2$, and $l$ from $2^3$ to $2^{13}$. For the low-rank cases, the corresponding TTD or HTD factors are provided as inputs. We compared the following three implementations:  
(i) the definition-based T-product; 
(ii) the proposed TTD-based T-product according to \cref{prop:ttd_tprod}; and (iii) the proposed HTD-based T-product framework according to \cref{prop:htdtd_tprod}. For each problem size, the experiment was repeated five times, and the average runtime was recorded. The results in \cref{fig:ttd_tprod} show that the proposed TTD- and HTD-based methods consistently outperform the definition-based implementation across all three data generation strategies. The performance advantage becomes more pronounced as the tensor dimensions increase, especially for tensors with low TT-rank or hierarchical rank, where computations are performed directly on compressed representations. The HTD-based method is generally faster than the TTD-based method across all three data generation strategies because it first contracts the shared physical mode in the reduced representation, whereas the TTD-based construction multiplies the TT-ranks in the corresponding output core.
\begin{figure}[t]
    \centering
    \includegraphics[width=1\linewidth]{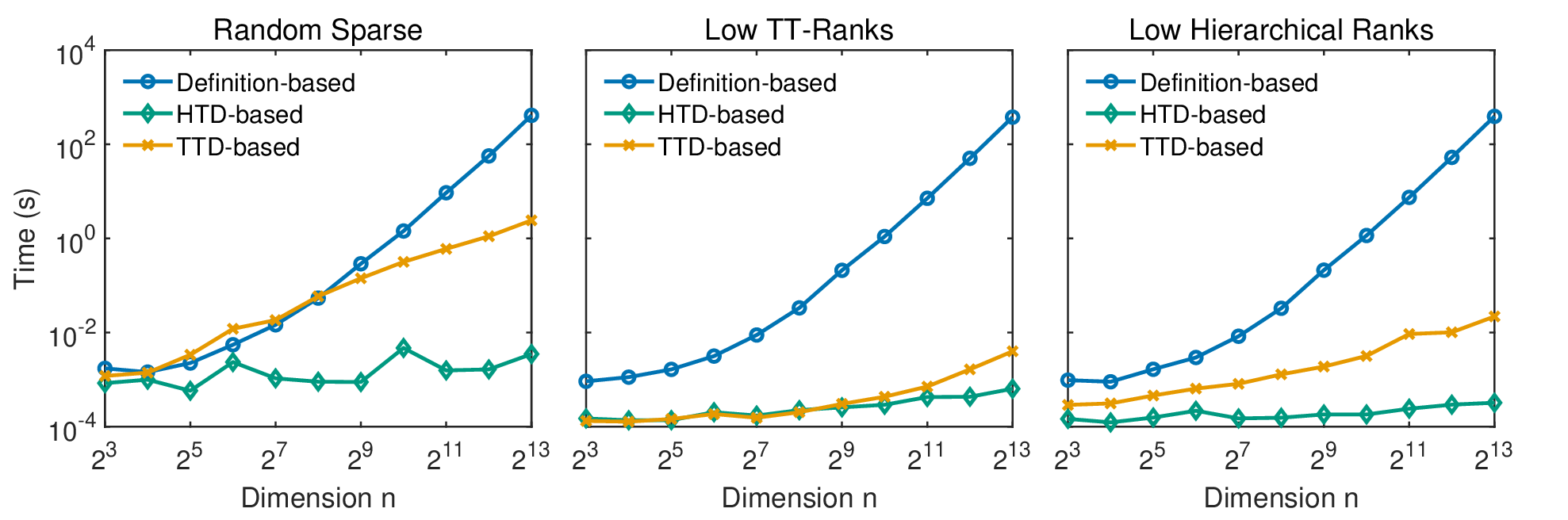}
     \vspace{-20pt}
    \caption{Log-log plots of the average runtime for the definition-based, TTD-based, and HTD-based T-product under three data generation strategies: random sparse tensors,  random tensors with low TT-rank, and random tensors with low hierarchical rank.}
   
    \label{fig:ttd_tprod}
\end{figure}

\subsection{Fourth-order SVD} \label{sec:num_highorder_tsvd}

We evaluated the computational efficiency of the proposed TTD- and HTD-based frameworks for computing the fourth-order SVD. We considered fourth-order tensors $\mathscr{T}\in\mathbb{R}^{n_1\times n_2\times n_3\times n_4}$ generated using the same three data generation strategies as in the previous experiment. The spatial dimensions were varied as $n_1=n_2=2^p$, $p=3,4,\dots,13$, while the transform-mode dimensions were fixed at $n_3=n_4=4$. We compared three implementations: (i) the definition-based higher-order SVD; (ii) the proposed TTD-based higher-order SVD according to \cref{thm:ttd_kth_tsvd_core}; and (iii) the proposed HTD-based higher-order SVD according to \cref{thm:htd_kth_tsvd_core}. For each problem size, the experiment was repeated five times, and the average runtime was recorded. As shown in \cref{fig:ttd_tsvd}, both the proposed TTD- and HTD-based methods consistently outperform the definition-based implementation across all three data generation strategies. The performance advantage becomes more pronounced as the tensor dimensions increase, particularly when the tensors admit low TT-rank or HTD representations. Moreover, in additional tests beyond the range shown in \cref{fig:ttd_tsvd}, we found that the definition-based method failed to complete the test cases for $p\geq 15$ due to its substantially higher memory and computational requirements, whereas the proposed methods remained computationally feasible.

\begin{figure}[t]
    \centering
    \includegraphics[width=1\linewidth]{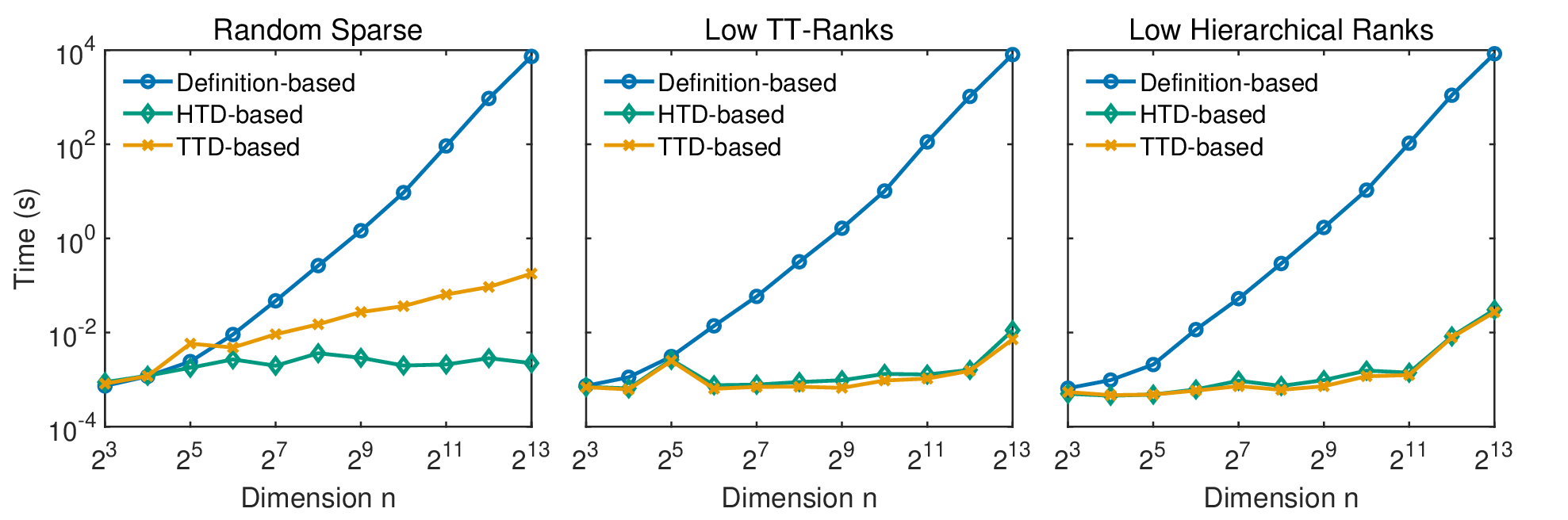}
    \vspace{-20pt}
    \caption{Log-log plots of the average runtime for the definition-based, TTD-based, and HTD-based higher-order SVD methods under three data generation strategies: random sparse tensors, random tensors with low TT-rank, and random tensors with low hierarchical rank.}
    \label{fig:ttd_tsvd}
\end{figure}

\subsection{T-eigensystem realization algorithm}
The proposed framework is particularly beneficial for large-scale applications in which the T-SVD constitutes the dominant computational kernel. As an example, we considered the tensor eigensystem realization algorithm (T-ERA) \cite{mei2026data}, which identifies reduced-order multilinear systems from measured data using the T-product and T-SVD. T-ERA generalizes the classical eigensystem realization algorithm (ERA) \cite{Juang1985era} to multilinear systems through transform-based tensor algebra. Since the computational cost of T-ERA is dominated by the T-SVD of a large generalized Hankel tensor, it provides a natural benchmark for evaluating the proposed framework. 

We replaced the definition-based T-SVD in T-ERA with the proposed TTD- and HTD-based formulations and compared their computational performance. 
Specifically, we benchmarked T-ERA using synthetic generalized Hankel tensors with one of three prescribed structures: (a) random sparse, (b) low TT-rank, or (c) low hierarchical rank, following the framework in \cite{mei2026data}, and the Hankel dimensions satisfy $\ell(L+1)=m(T+1)=H$, where we fixed $H=10,000$. We compared three implementations of T-ERA that differ only in the computation of the Hankel T-SVD: (i) the baseline T-ERA using the definition-based T-SVD; (ii) T-ERA with the TTD-based T-SVD; and (iii) T-ERA with the HTD-based T-SVD. For each configuration, the experiment was repeated ten times, and the average runtime for reduced model construction was recorded. As shown in \cref{fig:tera}, replacing the definition-based Hankel T-SVD with the proposed TTD- or HTD-based formulations substantially reduces the overall runtime of T-ERA across all three data generation strategies. 
The relative $\mathcal{H}_\infty$ errors between the reduced systems generated by the TTD-/HTD-based methods and the definition-based method remain below $2\times10^{-14}$ for all three test cases (see \cref{tab:tera_relative_errors}). These results confirm that the proposed T-SVD methods effectively alleviate the dominant computational bottleneck in large-scale T-ERA.

\begin{figure}[t]
    \centering
    \includegraphics[width=1\linewidth]{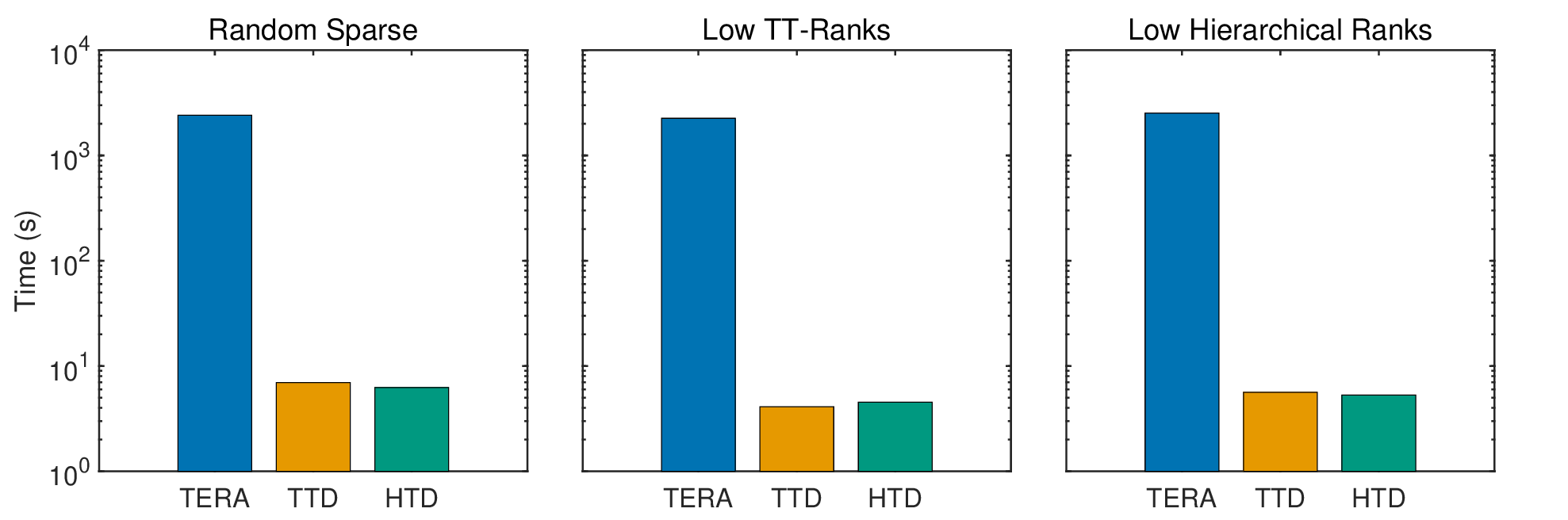}
    \vspace{-20pt}
    \caption{Log plots of the average runtime of the baseline, TTD-based, and HTD-based T-ERA  for generalized Hankel tensors with random sparse,  low TT-rank, and low hierarchical rank structures.}
    \label{fig:tera}
\end{figure}

\begin{table}[htbp]
\centering
\caption{Relative $\mathcal{H}_\infty$ errors of the reduced systems obtained by the TTD- and HTD-based T-ERA with respect to the definition-based T-ERA model for the three test cases.}
\label{tab:tera_relative_errors}
\resizebox{0.6\textwidth}{!}{
\begin{tabular}{lcc}
\toprule
Case
& TTD-based
& HTD-based \\
\midrule
Random sparse
& $5.77\times10^{-15}$
& $ 5.97\times10^{-15}$ \\
Low TT-ranks
& $9.05\times10^{-15}$
& $9.16\times10^{-15}$ \\
Low hierarchical ranks
& $1.67\times10^{-14}$
& $1.58\times10^{-14}$ \\
\bottomrule
\end{tabular}}
\end{table}

\section{Conclusion} \label{sec:conclusion}


In this article, we developed decomposition-based computational frameworks for transform-based multilinear algebra using TTD and HTD. They formulate block diagonalization, the T-product, and T-SVD directly in compressed representations, so that the computations are performed on TTD core tensors or HTD factors without full-tensor reconstruction. We further extended the frameworks to higher-order tensors and demonstrated their effectiveness numerically, including in multilinear system model reduction. The results show that the proposed TTD-based and HTD-based methods can substantially reduce computational cost while preserving the operator-based structure of transform tensor algebra, especially when the tensors have low TT- or hierarchical ranks.

Several directions merit further investigation. The proposed decomposition-based framework can be extended to other transform-based tensor operations, including T-QR, tensor inverse computations, and additional tensor factorizations. It would also be valuable to move beyond the discrete Fourier transform by incorporating more general invertible transforms together with a broader class of tensor decompositions, enabling the computational representation to better adapt to the underlying data and operators. This includes the development of transform-dependent compressed formulations as well as practical strategies for selecting appropriate decomposition formats for different applications. More broadly, the proposed framework can be integrated into computational pipelines in which transform-based tensor operations arise repeatedly as fundamental computational kernels, including multilinear system identification and model reduction, multidimensional imaging and video analysis, hyperspectral data processing, parametric partial differential equations and uncertainty quantification, and scientific machine learning. We expect that these directions will further expand the scope of scalable transform-based multilinear algebra and its applications to large-scale tensor computation.

\bibliographystyle{siamplain}
\bibliography{references}
\end{document}